\documentclass[final,onefignum,onetabnum]{siamart190516}

\usepackage{amsfonts}
\usepackage{graphicx}
\usepackage{epstopdf}
\usepackage{scalerel}
\usepackage{framed,multirow}
\usepackage{etoolbox}
\usepackage{amscd,amsmath,amssymb,mathrsfs,bbm,listings}
\usepackage{tikz} 
\usetikzlibrary{arrows,automata,calc}
\usepackage{epsfig}
\usepackage{multirow}
\usepackage{subfig}

\numberwithin{equation}{section}

\ifpdf
  \DeclareGraphicsExtensions{.eps,.pdf,.png,.jpg}
\else
  \DeclareGraphicsExtensions{.eps}
\fi

\def\msquare{\mathord{\scalebox{0.4}[0.4]{\scalerel*{\Box}{\strut}}}}

\newsiamremark{remark}{Remark}
\newsiamthm{condition}{Condition}

\def\dx{\,\mathrm{d}\bx} 
\def\dS{\,\mathrm{d}S}
\def\dV{\,\mathrm{d}V}

\newcommand\REL[1]{{\mathfrak{Re}}\left({#1}\right)}
\newcommand\IMG[1]{\mathrm{Im}\left({#1}\right)}
\newcommand\DIM[1]{\mbox{dim}\left({#1}\right)}
\newcommand\RANK[1]{\mbox{rank}\left({#1}\right)}

\newcommand{\dfdt}[1]{\frac{\partial {#1}}{\partial t}}
\newcommand{\dfhdt}[1]{\frac{d {#1}}{d t}}

\newcommand{\EFC}[2]{\calC^{#1}\left({#2}\right)}
\newcommand{\ESOBOLEV}[2]{H^{#1}\left({#2}\right)}
\newcommand{\PartDer}[2]{\partial^{#1}_{#2}}
\newcommand{\DerA}[2]{D^{#1}{#2}}
\newcommand{\Taylor}[3]{T_{#1}^{#2}\left[#3\right]}
\newcommand{\AvTaylor}[2]{\calQ^{#1}\left[#2\right]}
\newcommand{\Conjugate}[1]{\overline{#1}}
\newcommand{\jump}[1]{\left[\!\left[#1\right]\!\right]}

\numberwithin{equation}{section}

\newcommand{\vnSpace}[1]{{{\vn}_{#1}^x}}
\newcommand{\vnTime}[1]{{n_{#1}^t}}
\newlength{\dhatheight}
\def\QFLUX{
    \settoheight{\dhatheight}{\ensuremath{\widehat{q_h}}}
    \addtolength{\dhatheight}{-0.40ex}
    \widehat{\vphantom{\rule{1pt}{\dhatheight}}
    \smash{\widehat{\nabla \psi}}}_{hp}
}

\allowdisplaybreaks[4]
\newcommand{\mvec}[1]{{\vec{\mathbf{#1}}}}
\newcommand{\vn}{{\mvec n}}
\newcommand{\uu}[1]{\hbox{\boldmath$#1$}}
\newcommand{\Uu}[1]{{\mathbf{#1}}}
\newcommand{\bc}{{\Uu c}}
\newcommand{\bj}{{\boldsymbol j}}
\newcommand{\bjx}{{\boldsymbol j_\bx}}
\newcommand{\bC}{{\Uu C}}
\newcommand{\bD}{{\Uu D}}
\newcommand{\bG}{{\Uu G}}
\newcommand{\bP}{{\Uu P}}
\newcommand{\bS}{{\Uu S}}
\newcommand{\bT}{{\Uu T}}
\newcommand{\bV}{{\Uu V}}

\newcommand{\bzero}{\Uu{0}}
\newcommand{\calC}{{\mathcal C}}
\newcommand{\calD}{{\mathcal D}}
\newcommand{\calT}{{\mathcal T}}
\newcommand{\calE}{{\mathcal E}}
\newcommand{\calF}{{\mathcal F}}
\newcommand{\calO}{{\mathcal O}}
\newcommand{\calQ}{{\mathcal Q}}
\newcommand{\Fh}{\calF_h}
\newcommand{\Kx}{K_{\bx}}
\newcommand{\Th}{{\calT_h}}

\newcommand{\Thxn}{{\calT_{h_{\bx,n}}^{\bx}}}

\newcommand{\deK}{{\partial K}}
\newcommand{\GD}{{\Gamma_{\mathrm D}}}
\newcommand{\GN}{{\Gamma_{\mathrm N}}}
\newcommand{\GR}{{\Gamma_{\mathrm R}}}
\newcommand{\uhp}{{\psi_{hp}}}
\newcommand{\uhpT}{{\psi_{hp}^{-}}}
\newcommand{\shp}{{s_{hp}}}
\newcommand{\cs}{{\conj{s_{hp}}}}

\newcommand{\IN}{\mathbb{N}}
\newcommand{\IR}{\mathbb{R}}
\newcommand{\po}{\partial \Omega}
\newcommand{\gD}{{g_{\mathrm D}}}
\newcommand{\gN}{{g_{\mathrm N}}}
\newcommand{\OO}{{(\Omega)}}
\newcommand{\oon}{\;\text{on}\;}
\newcommand{\bx}{{\Uu x}}
\newcommand{\bn}{{\Uu n}}
\newcommand{\bb}{{\Uu b}}
\newcommand{\bN}{{\Uu N}}
\newcommand{\bM}{{\Uu M}}
\newcommand{\bz}{{\Uu z}}
\newcommand{\btau}{{\boldsymbol\tau}}		
\newcommand{\IC}{\mathbb{C}}
\newcommand{\IP}{\mathbb{P}}
\newcommand{\IT}{\mathbb{T}}
\newcommand{\IZ}{\mathbb{Z}}
\newcommand{\tfor}{\text{ for }}
\newcommand*{\jmp}[1]{[\![#1]\!]}                     
 \newcommand{\mvl}[1]{\left\{ \!\left\{#1\right\}\!\right\}}  
 \newcommand{\FT}{{\Fh^T}}
\newcommand{\FO}{{\Fh^0}}

\newcommand{\FD}{{\Fh^{\mathrm D}}}
\newcommand{\FN}{{\Fh^{\mathrm N}}}
\newcommand{\FR}{{\Fh^{\mathrm R}}}
\newcommand{\rtime}{{\mathrm{time}}}
\newcommand{\rspace}{{\mathrm{space}}}
\newcommand{\Fspa}{{\Fh^\rspace}}
\newcommand{\Ftime}{{\Fh^\rtime}}
\newcommand{\bVp}{{\IT_p}} 
\newcommand*{\conj}[1]{\overline{#1}}
\newcommand{\hp}{_{hp}}
\newcommand\hUhp{\widehat\psi\hp}
\DeclareMathOperator{\re}{{\mathfrak{Re}}} %
\DeclareMathOperator{\im}{{\mathfrak{Im}}} %
 
\DeclareMathOperator{\esssup}{ess\,sup}
\DeclareMathOperator{\essinf}{ess\,inf}
\DeclareMathOperator{\spn}{span}
\newcommand{\cbA}[2]{{{\mathcal{A}}}\left({#1};\ {#2}\right)}
\newcommand*{\N}[1]{\left\|#1\right\|}
\newcommand*{\abs}[1]{\left|#1\right|}
\newcommand{\Tnorm}[2]{|||#1|||_{#2}}
\newcommand{\MatrixNorm}[2]{|||#1|||_{#2}}
\newcommand*{\Norm}[2]{\left\|#1\right\|_{#2}}
\newcommand{\DG}{_{\mathrm{DG}}}
\newcommand{\DGp}{_{\mathrm{DG^+}}}
\newcommand{\di}{\,\mathrm{d}}                              
\newcommand*{\der}[2]{\frac{\partial #1}{\partial #2}}
\newcommand{\derp}[2]{{\frac{\partial #1}{\partial #2}}}
\newcommand{\cw}{{\conj w}}
\newcommand{\deO}{{\partial\Omega}}
\newcommand{\cS}{{\mathcal S}}
\newcommand{\ba}{{\Uu a}}
\newcommand{\bd}{{\Uu d}}
\newcommand{\bv}{{\Uu v}}
\newcommand{\ee}{{\rm e}}
\newcommand{\ri}{{\rm i}}
\newcommand{\mj}{{\boldsymbol{j}}}
\newcommand{\tr}{\mathrm{tr}}
\newcommand{\tand}{\text{ and }}

\headers{A space--time Trefftz-DG method for the LSE}{S. G\'omez and A. Moiola}

\title{A space--time Trefftz discontinuous Galerkin method for the linear 
Schr\"odinger equation
\thanks{\funding{A.\ Moiola acknowledges support from GNCS--INDAM, from PRIN project ``NA-FROM-PDEs'' and from MIUR through the ``Dipartimenti di Eccellenza'' Program (2018--2022) --- Dept.\ of Mathematics, University of Pavia.}}}
 
\author{Sergio G\'omez\thanks{Dipartimento di Matematica ``F.~Casorati'', Universit\`a di Pavia, 27100 Pavia, Italy.\newline
(\email{sergio.gomez01@universitadipavia.it}, \email{andrea.moiola@unipv.it}).}
\and Andrea Moiola\footnotemark[2]}

\begin{document}		
\maketitle

\begin{abstract}
A space--time Trefftz discontinuous Galerkin method for the Schr\"odinger equation with piecewise-constant potential is proposed and analyzed. 
Following the spirit of Trefftz methods, trial and test spaces are spanned by
non-polynomial complex wave functions that satisfy the Schr\"odinger equation locally on each element of the space--time 
mesh.
This allows {for a significant reduction in} the number of degrees of freedom in 
comparison with full polynomial spaces. 
We prove well-posedness and stability of the method, and, for the one- and 
two-dimensional cases, optimal, high-order, $h$-convergence error estimates in a 
skeleton norm.
Some numerical experiments validate the theoretical results presented.
\end{abstract}



\section{Introduction\label{SEC::INTRODUCTION}}

In this work we consider the following initial boundary value problem for the 
homogeneous, time-dependent Schr\"odinger equation on a space--time 
cylinder $Q = \Omega \times I$, where $\Omega$ is an open and 
bounded domain in $\IR^d$, $d\in\IN$, with Lipschitz boundary $\po$ and $I = 
(0, T)$, for 
some $T > 0$:
\begin{subequations}
\label{EQN::SCHRODINGER-EQUATION}
\begin{align}
i\dfdt{\psi} + \Delta \psi - V \psi & = 0, \quad \quad\; \mbox{ in }\ Q, 
\label{EQN::SCHRODINGER-EQUATION-1} \\
\psi & = \gD, \quad \ \   \mbox{ on }\ \po \times I,
\label{EQN::SCHRODINGER-EQUATION-2}\\
\psi(\bx, 0) & = \psi_0(\bx), \ \mbox{ on }\ \Omega.
\label{EQN::SCHRODINGER-EQUATION-3}
\end{align}
\end{subequations}
Here the Dirichlet boundary datum $\gD$ and the initial condition $\psi_0$ 
are given functions; $V: \Omega \rightarrow \IR$ is a piecewise-constant 
potential and the Laplacian operator $\Delta$ refers to the space 
variable $\bx$ only. 
Problem \cref{EQN::SCHRODINGER-EQUATION} is well-posed if, e.g., $\psi_0\in 
H^1_0\OO$ and $\gD=0$, by \cite[Chapter 3, Thm.~10.1, 
Rem.~10.2]{Lions_Magenes_1972};
in this case $\psi\in C^0(0,T; H^1_0\OO)\cap C^1(0,T; H^{-1}\OO)$.

The model \cref{EQN::SCHRODINGER-EQUATION} arises from a wide number of 
applications: it is the fundamental equation of quantum mechanics 
\cite{Lifshitz_Landau_1965}, in optics it is known as ``paraxial wave
equation'' and approximates the Helmholtz equation when 
the optical field acts mostly along one specific axis (Fresnel's approximation)
\cite{Grella_1982}, while in underwater acoustics it is called ``parabolic 
equation'' \cite{Keller_Papadakis_1977}.

The aim of this work is to propose and analyze a space--time Trefftz-DG
method 
for the numerical solution of \cref{EQN::SCHRODINGER-EQUATION}. The main feature 
of Trefftz methods is that they seek approximations in spaces 
spanned by local solutions of the partial differential equation considered. 
This typically requires non-polynomial basis functions. Trefftz schemes are 
mainly motivated by their significant reduction in the computational cost and 
number of degrees of freedom with respect to traditional polynomial 
approximations, and by their effectiveness in dealing with the intrinsic highly 
oscillatory behavior in the solution of certain problems.
On the other hand, Discontinuous Galerkin (DG) is a class of finite element 
methods that do not impose continuity in a strong sense on the approximation, 
making them specially suitable to be combined with Trefftz bases, which are 
naturally discontinuous.

Trefftz-DG methods have been successfully derived for many important 
equations; among others, the Helmholtz equation 
\cite{Gittelson_Hiptmair_Perugia_2009,Hiptmair_Moiola_Perugia_2016}, the time-harmonic and time-dependent Maxwell's equations  
\cite{Egger_Kretzchmar_Scnepp_Weiland_2015,Hiptmair_Moiola_Perugia_2013}, 
the second- and first-order formulations of the acoustic wave equation 
\cite{Banjai_Georgoulis_Lijoka_2017,Moiola_Perugia_2018}. However, to the 
best of our knowledge, this is the first attempt to 
study the application of a Trefftz-DG method to the linear Schr\"odinger 
equation.

The well-posedness and quasi-optimality error analysis of the Trefftz-DG scheme 
closely follows the analysis previously developed for the wave 
\cite[\S5.2]{Moiola_Perugia_2018} and the Helmholtz 
\cite[\S2.2.1]{Hiptmair_Moiola_Perugia_2016} equations.
This part admits the use of any discrete Trefftz space.
The chosen DG formulation admits general polytopic space meshes that 
are shape-regular, locally quasi-uniform {and aligned to the discontinuities in the potential $V$.}
Inverse estimates are not needed in the analysis: this is a strong advantage in 
comparison to other DG formulations (such as interior-penalty, see 
\cite{Banjai_Georgoulis_Lijoka_2017} for a Trefftz example) because 
 inverse estimates for non-polynomial discrete Trefftz spaces are in 
general hard to obtain (see \cite[\S3.2]{Gittelson_Hiptmair_Perugia_2009}).

Differently from the acoustic wave equation \cite{Moiola_Perugia_2018}, the presence of derivatives of different orders in \eqref{EQN::SCHRODINGER-EQUATION-1} prevents the existence of non-trivial polynomial solutions of the Schr\"odinger equation, so we construct Trefftz basis functions as simple complex exponentials \eqref{EQN::BASIS-FUNCTIONS}.
In order to establish convergence rates in the mesh size $h$, the key ingredient is the analysis of the approximation properties of carefully designed discrete spaces.
The key idea was introduced by O.~Cessenat and B.~Despr\'es in the proof of 
\cite[Thm.~3.7]{Cessenat_Despres_1998} (in the case of the ultra weak 
variational formulation (UWVF) applied to the Helmholtz equation): given any 
smooth PDE solution $\psi$, 
if the local discrete space contains an element with the same degree-$p$ Taylor polynomial of $\psi$, then the space enjoys the same $h$-approximation properties of the space $\IP^p$ of degree-$p$ polynomials.
We prove that this condition is satisfied by a simple discrete space $\IT^p$ in low space dimensions $d=1$ {and $d=2$ (we comment on the case $d\ge3$ in \Cref{Rem:3D})}.
Additional difficulties are due to the fact that we only assume Sobolev 
regularity of the PDE solution, so the Taylor polynomial has to be understood in 
an ``averaged'' sense.
It turns out that $\dim\IT^p\ll\dim\IP^p$: the Trefftz scheme {allows for} much faster 
convergence in terms of degrees of freedom than classical polynomial DG schemes, 
see \cref{Rem:LessDOFs}.
This approach to the Trefftz approximation theory is completely different from that 
used for the Helmholtz equation in \cite{Moiola_Hiptmair_Perugia_2011}, which is 
based on the use of an integral (Vekua) transform and circular wave expansions.
{A shortcoming of the Taylor-polynomial approach of \cite{Cessenat_Despres_1998} and the present paper, in contrast to the Vekua-transform technique of \cite{Moiola_Hiptmair_Perugia_2011}, is that the approximation analysis does not extend to the $p$-convergence case.
This is a difficult task that, in the context of Trefftz schemes, has been achieved for time-harmonic equations but not yet for the wave equation.
}

The paper is structured as follows: in \Cref{SECT::TREFFTZ-DG}
we introduce some standard notation and present the proposed Trefftz-DG method 
whose numerical fluxes are chosen as upwind in time and classical average in 
space with an appropriate complex penalization. 
In \Cref{SECT::WELL-POSEDNESS} we prove the well-posedness and quasi-optimality 
of the Trefftz-DG approximation for arbitrary dimensions and 
discrete Trefftz subspaces. 
\Cref{SECT::ERROR-ESTIMATE} is devoted to the error analysis for discrete 
subspaces spanned by complex exponentials 
satisfying the Schr\"odinger equation. In \Cref{SECT::GENERAL-APPROXIMATION} 
we 
present a condition which guarantees optimal local approximation of the exact 
solution in a general discrete Trefftz space; assuming that such 
condition 
is satisfied we prove an $h$-estimate in a mesh-skeleton norm in \Cref{SECT::DG-ERROR-ESTIMATE}.
In \Cref{SECT::APPROX-1D,SECT::APPROX-2D} 
we prove that the assumed condition is indeed true for the $(1 + 1)$ and $(2 + 
1)$ dimensional cases under some restrictions of the tuning parameters for our 
basis choice.
Some numerical experiments validating our theoretical results are 
presented in \Cref{SECT::NUMERICAL-EXPERIMENTS}. We propose some possible future 
extensions of the method and its analysis
in \Cref{SECT::CONCLUSIONS}.


\section{Trefftz-discontinuous Galerkin method\label{SECT::TREFFTZ-DG}}
\subsection{Mesh and DG notation}\label{S:Mesh}
Let the time interval $(0, T)$ be partitioned as
\begin{gather*}
 0\ =\ t_0\ <\ t_1\ <\ \ldots\ <\ t_N\ =\ T,  \\
I_n := (t_{n-1}, t_n),  \qquad h_n:= t_n - t_{n-1}, \qquad 
h_t: = \max\limits_{1 \leq n \leq N} h_n.
\end{gather*}
We denote the time-slabs $D_n:= \Omega \times I_n$.
For each $n =1, \ldots, N$, we assume to have a 
polytopic partition $\Thxn = \left\{\Kx\right\}$ of $\Omega$ such that for all 
$K_{\bx} \in \Thxn$, the restriction of the potential $V$ to $K_{\bx}$ is 
constant and
\[
h_{K_{\bx}}:= \mbox{diam}(K_{\bx}), \qquad
h_\bx:=\max_{K_\bx\in\Thxn,\; n=1,\ldots,N} h_{K_\bx}.
\]
Additionally we also assume that each $\Thxn$ satisfies the following 
properties:
\begin{itemize}
\item {\bf Shape-regularity}: there exists a number $\mathsf{sr}(\Th)>0$ such that 
$h_{\Kx}\le \rho_{\Kx}\mathsf{sr}(\Th)$ for all elements $\Kx\in\Thxn$, where $\rho_{K_\bx}$ is the radius of a $d$-dimensional ball contained in $K_{\bx}$.
\item {\bf Local quasi-uniformity in space}: there exists a number 
$\mathsf{lqu}(\Th)>0$ such that $h_{\Kx^1}\le h_{\Kx^2}\, \mathsf{lqu}(\Th)$ for 
all $n=1,\ldots,N$ and $K_\bx^1,K_\bx^2\in\Thxn$ such that $K_\bx^1\cap K_\bx^2$ 
has positive $(d-1)$-dimensional measure.
\end{itemize}
We define the space--time finite element mesh 
$$\Th(Q) := \Big\{K = \Kx \times I_n\ : \ \Kx\in \Thxn,\ n = 1, \ldots, N \Big\}.$$

Each internal mesh face $F$ (i.e.\ any $F=\deK_1\cap \deK_2$, for 
$K_1,K_2\in\Th(Q)$, with positive $d$-dimensional measure) is either
\begin{equation*}
\begin{cases}
\text{a space-like face:} & \text{if } F\subset\Omega\times\{t_n\}, \; \tfor 0< n< N, \text{ or}\\
\text{a time-like face:} & \text{if } F\subset \partial \left(\Kx^1\times 
I_n\right)  \cap\partial \left(\Kx^2\times I_n\right), 
\;\\& \qquad 
\tfor \Kx^1,\Kx^2\in\Thxn,\; 1\le n\le N.
\end{cases}
\end{equation*}
We denote the mesh skeleton and its parts as
\begin{align*}
\Fh :=& \bigcup_{K \in \Th(Q)} \deK,\qquad
\FO := \Omega \times \left\{0\right\}, \quad 
\FT := \Omega \times \left\{T\right\}, \quad
\FD := \po \times (0, T),
\\
\Ftime :=& \mbox{ the union of all the time-like faces},\\
\Fspa  :=& \mbox{ the union of all the space-like faces}.
\end{align*}

We also employ the standard DG notation for the averages $\mvl{\cdot}$ and 
space $\jump{\cdot}_{\bN}$ and time $\jump{\cdot}_t$ jumps for 
piecewise-continuous complex scalar $w$ and vector $\btau$ fields:
\begin{align*}
&\begin{cases}
\mvl{w} : = \frac{1}{2} \left(w|_{K_1} + w|_{K_2}\right)\\
\mvl{\btau} : = \frac{1}{2} \left(\btau|_{K_1} + \btau|_{K_2}\right)
\end{cases}
&&\oon \deK_1\cap\deK_2\subset\Ftime,
\\
&\begin{cases}
\jump{w}_\bN : = w|_{K_1} \vnSpace{K_1} + w|_{K_2} \vnSpace{K_2}\\
\jump{\btau}_\bN : = \btau|_{K_1} \cdot \vnSpace{K_1} + \btau|_{K_2} \cdot \vnSpace{K_2}
\end{cases}
&&\oon\deK_1\cap\deK_2\subset\Ftime,
\\
&\begin{cases}
\jump{w}_t : = w^- - w^+\\
\jump{\btau}_t : = \btau^- - \btau^+,
\end{cases}
&&\oon\Fspa,
\end{align*}
where $\vnSpace{K}\in\IR^d$ is the space component of the outward-pointing unit 
normal vector on $\deK\cap\Ftime $, and the superscripts ``$-$'' and ``$+$'' 
are used to 
denote the traces on $\Omega\times\{t_n\}$ of scalar and vector 
fields from the time-slabs $D_n$ and $D_{n+1}$, at lower and higher times, 
respectively.

\subsection{Formulation of the Trefftz-DG method}\label{S:DG}
We define the local and global Trefftz spaces:
\begin{align}
\nonumber
\bT(K) & := \Big\{ w \in \ESOBOLEV{1}{I_n; L^2(\Kx)} \cap L^2\left(I_n; 
\ESOBOLEV{2}{\Kx}\right) \; \text{ such that}  \\
&\hspace{25mm} i \dfdt{w} + \Delta w 
- V w = 0 \ \mbox{ on }K = \Kx \times I_n\Big\}, 
\label{EQN:TrefftzSpace}\\
\bT(\Th) &:= \left\{w \in L^2\left(Q\right)^{d+1}\ \Big|\ w|_K \in \bT(K), \ 
\forall K \in \Th(Q)\right\}.
\nonumber
\end{align}
For any finite-dimensional subspace $\bVp\left(\Th\right) \subset \bT(\Th)$ the 
proposed Trefftz-DG method applied to \cref{EQN::SCHRODINGER-EQUATION} seeks 
an 
approximation $\uhp(\bx, t)\in \bVp(\Th)$ of the exact solution $\psi(\bx, t) 
\in \bT(\Th)$ such that for any test function $\shp \in \bVp(\Th)$ the 
following equation is satisfied for all $K \in \Th(Q)$
\begin{align}
\int_K \uhp & \Big(\conj{i \dfdt{\shp} + \Delta \shp - V \shp} \Big) \dV 
\nonumber \\
 & + 
\oint_{\partial K}\left[ i \hUhp  \conj\shp \vnTime{K} + \Big( \QFLUX 
\conj\shp - 
\hUhp\nabla \conj\shp\Big) \cdot \vnSpace{K}\right]\dS = 0, 
\label{EQN::TREFFTZ-DG-FORMULATION}
\end{align}
where $\conj{\,\cdot\,}$ denotes the complex conjugate, $\vnTime{K}$ is the time 
component of the outward-pointing unit normal vector on $\deK$ (so 
$\vnTime{K}=\pm1$ on $\deK\cap(\Fspa\cup\FO\cup\FT)$ and $\vnTime{K}=0$ 
otherwise).
\Cref{EQN::TREFFTZ-DG-FORMULATION} is obtained integrating by parts 
the product of \cref{EQN::SCHRODINGER-EQUATION-1} and $\conj\shp$ twice in 
space and once in time, noting the sign change in the time-derivative term due 
to the conjugation of $i$.
The so-called \textit{numerical fluxes} $\hUhp$ and $\QFLUX $ are approximations 
of the traces 
of $\uhp$ and $\nabla \uhp$ on $\Fh$. We choose them as:
\begin{align*}
\label{EQN::NUMERICAL-FLUXES}
	\hUhp & := \left\{
	\begin{tabular}{ll}
	$\uhpT$, & on $\Fspa$,\\[1ex]
	$\uhp$, & on $\FT$,\\[1ex]
	$\psi_0$, & on $\FO$, \\[1ex]
	$\mvl{\uhp} - i\beta \jump{\nabla \uhp}_{\bN}$, & on $\Ftime$,\\[1ex]
	$\gD$, & on $\FD$,
	\end{tabular}
	\right. \\[1mm]
	\QFLUX & := \left\{
	\begin{tabular}{ll}
	$\mvl{\nabla \uhp} + i\alpha \jump{\uhp}_{\bN}$, & on 
$\Ftime$,\\[1ex]
	$\nabla \uhp  + i \alpha \left(\uhp - \gD \right)\vnSpace{\Omega}$, & on 
$\FD$,
	\end{tabular}
	\right.
\end{align*}
where $\alpha \in L^{\infty}(\Ftime \cup \FD)$ and $\beta \in 
L^{\infty}(\Ftime)$ are some mesh-dependent stabilization parameters 
with $\essinf_{\Ftime\cup\FD}\alpha>0$ and $\essinf_\Ftime\beta>0$. {These stabilization parameters play an important role in the convergence of the method; in \cref{THM::ERROR-ESTIMATE} we present a choice that ensures optimal $h$-convergence in a mesh-skeleton norm.}

Since all $\shp \in \bVp(\Th)$ satisfy \cref{EQN::SCHRODINGER-EQUATION-1} in 
each mesh element, the volume integral in \Cref{EQN::TREFFTZ-DG-FORMULATION} 
vanishes and 
thus the Trefftz-DG method only involves integrals on the mesh skeleton.
Consequently, after summing \Cref{EQN::TREFFTZ-DG-FORMULATION} over 
all the elements $K \in 
\Th(Q)$ and substituting the definition of the numerical fluxes, the following 
Trefftz-DG variational formulation is obtained:
\begin{equation}
\label{EQN::VARIATIONAL-TREFFTZ-DG}
\mbox{Seek }\uhp \in \bVp(\Th) \mbox{ such that: } \cbA{\uhp}{\shp} = 
\ell(\shp), \quad \forall \shp \in \bVp(\Th),
\end{equation}
where
\begin{align*}
\cbA{\uhp}{\shp} & := \int_{\Fspa}i \uhpT \jump{\cs}_t  \dx + \int_{\FT} i 
\uhp \cs \dx \\
&\; + \int_{\Ftime} \Big(\mvl{\nabla \uhp} \cdot \jump{\cs}_{\bN} + i \alpha 
\jump{\uhp}_{\bN} \cdot \jump{\cs}_{\bN} - \mvl{\uhp} \jump{\nabla 
\cs}_{\bN}  \\
&\hspace{10mm} + i \beta 
\jump{\nabla \uhp}_{\bN} \jump{\nabla \cs}_{\bN}\Big) \dS + \int_{\FD} 
\left(\nabla \uhp \cdot \vnSpace{\Omega}+ i \alpha 
\uhp \right)\cs \dS  ,\\
\ell(\shp) & := \int_{\FO} i \psi_0 \cs \dx + \int_{\FD} \gD \left(\nabla 
\cs \cdot \vnSpace{\Omega} + i \alpha\cs\right) \dS.
\end{align*}
{As a result of the Trefftz property, the definitions of $\cbA{\cdot}{\cdot}$ and $\ell(\cdot)$ in the variational formulation \cref{EQN::VARIATIONAL-TREFFTZ-DG} are independent of the potential $V$, which has an effect only on the discrete space.}


\section{Well-posedness, stability and quasi-optimality of the Trefftz-DG 
method}
\label{SECT::WELL-POSEDNESS}
The theoretical results in this section are derived for arbitrary 
space dimension $d$ and are independent of the specific {choice} of the discrete subspace
$\bVp(\Th)$, {whose elements are local solutions of the Schr\"odinger equation  \cref{EQN::SCHRODINGER-EQUATION-1}}.

The following identities will be used
\begin{subequations}
\label{EQN::AV-JUMP-IDENTITIES}
\begin{align} 
\REL{w^- \jump{\cw}_t}-  \frac{1}{2} \jump{\abs{w}^2}_t & = 
\frac{1}{2 } \abs{\jump{w}_t}^2, \qquad \mbox{on }\Fspa, 
\label{EQN::AV-JUMP-IDENTITIES-1}\\
\mvl{w} \jump{\btau}_{\bN} + \mvl{\btau} \cdot \jump{w}_{\bN} & = \jump{w 
\btau}_{\bN}, \qquad \ \ \ \mbox{ on } \Ftime.
\label{EQN::AV-JUMP-IDENTITIES-3}
\end{align}
\end{subequations}
Recalling that the numerical flux parameters $\alpha$ and $\beta$ are positive, 
we 
define the following mesh-dependent semi-norms:
\begin{align}
\label{EQN::DG-NORMS}
\Tnorm{w}{\DG}^2  : =  &\, \Norm{\jump{w}_t}{L^2(\Fspa)}^2 + 
\frac{1}{2}\Norm{w}{L^2(\FT \cup \FO)}^2 + \Norm{\alpha^{1/2} 
w}{L^2(\FD)}^2\\
\nonumber
& + \Norm{\alpha^{1/2} 
\jump{w}_{\bN}}{L^2(\Ftime)^d}^2 + \Norm{\beta^{1/2} \jump{\nabla 
w}_{\bN}}{L^2(\Ftime)}^2,
\\
\nonumber
\Tnorm{w}{\DGp}^2  : = &\, \Tnorm{w}{\DG}^2 + \Norm{w^-}{L^2(\Fspa)}^2 + 
\Norm{\alpha^{-1/2}\mvl{\nabla w}}{L^2(\Ftime)^d}^2\\
& + \Norm{\alpha^{-1/2}\nabla 
w \cdot \vnSpace{\Omega}}{L^2(\FD)} + 
\Norm{\beta^{-1/2}\mvl{w}}{L^2(\Ftime)}^2.
\nonumber
\end{align}
Even though $\Tnorm{\cdot}{\DG}$ and $\Tnorm{\cdot}{\DGp}$ are just seminorms 
on $\ESOBOLEV{1}{\Th}$, the following lemma shows that they are indeed norms on 
$\bT(\Th)$. Furthermore, continuity and coercivity of the sesquilinear form 
$\cbA{\cdot\, }{\!\!\cdot}$ with respect to these norms are proven in 
\Cref{PROP::COERCIVITY,PROP::CONTINUITY}.
\begin{lemma}
\label{LEMMA::DG-NORMS} 
	$\Tnorm{\cdot}{\DG}$ and $\Tnorm{\cdot}{\DGp}$ are norms on 
$\bT(\Th)$.
\end{lemma}
\begin{proof}
It is enough to prove that $\Tnorm{v}{\DG} = 0$, $v\in \bT(\Th)$, implies $v = 
0$. 
Indeed, if $\Tnorm{v}{\DG} = 0$, then, by the definitions  
\eqref{EQN:TrefftzSpace} and \eqref{EQN::DG-NORMS} of the Trefftz space and the 
DG norm, $v\in H^1(0,T;L^2\OO)\cap L^2(0,T;H^1_0\OO)$.
Then, for all $w$ in the same space,
$$
0=\sum_{K\in\calT_h}\int_K\Big(i\der vt +\Delta v-Vv\Big)\conj w\dV
=\int_Q\Big(i\der vt \conj w-\nabla v\cdot \nabla\cw-Vv\conj w\Big)\dV
$$
by the Trefftz property, where the boundary terms on $\Ftime$ arising from the integration by parts cancel because $\jmp{\nabla v}_\bN=0$ is implied by $\Tnorm{v}{\DG} = 0$.
Moreover $v(\cdot,0)=0$.
This means that $v$ is {a} variational solution of the homogeneous Schr\"odinger problem \eqref{EQN::SCHRODINGER-EQUATION} (i.e.\ with $\gD=0$ and $\psi_0=0$).
By the uniqueness of the solution in \cite[Chapter 3, 
Thm.~10.1]{Lions_Magenes_1972}, it follows that $v=0$.
\end{proof}

\begin{proposition}[Coercivity]
\label{PROP::COERCIVITY} 
For all $w \in \bT(\Th)$ the following identity holds 
\begin{equation}
\label{EQN::COERCIVITY}
\im\big(\cbA{w}{w}\big) = \Tnorm{w}{\DG}^2.
\end{equation}
\end{proposition}
\begin{proof}
Elementwise integration by parts for $w \in \bT(\Th)$ gives the identity:
\begin{align}
0 & = \im \bigg(\sum_{K \in \Th(Q)}\int_K w \Big(\conj{i \dfdt{w} + \Delta w - V w}\Big) \dV \bigg) \nonumber \\
& = -\frac{1}{2} \sum_{K \in \Th(Q)} \int_K \dfdt{\abs{w}^2} \dV + 
\im \bigg(\int_{\Ftime} \jump{w \nabla \conj{w}}_{\bN} \dS + \int_{\FD} w \nabla 
\conj{w} \cdot \vnSpace{\Omega}\dS\bigg) \nonumber \\
& = -\frac{1}{2} \bigg(\int_{\Fspa} \jump{\abs{w}^2}_t \dx + 
\int_{\FT}\abs{w}^2 \dx - \int_\FO  \abs{w}^2 \dx \bigg) \nonumber \\
& \qquad + \im \bigg(\int_{\Ftime} \jump{w \nabla \conj{w}}_{\bN} \dS 
+ \int_{\FD} w \nabla \conj{w} \cdot \vnSpace{\Omega}\dS\bigg).
\label{EQN::COERCIVITY-PROOF-1}
\end{align}
Together with the jump identities in \cref{EQN::AV-JUMP-IDENTITIES}, this gives
\begin{align*}
\im\left(\cbA{w}{w}\right)
& = \im \bigg(\cbA{w}{w} + \sum_{K \in \Th(Q)}\int_K w 
\Big(\conj{i \dfdt{w} + \Delta w - V w}\Big) \dV\bigg) \\
& \stackrel{\cref{EQN::COERCIVITY-PROOF-1}}{=} \int_{\Fspa} 
\Big(\REL{w^- \jump{\conj{w}}_t} - \frac{1}{2} 
\jump{\abs{w}^2}_t\Big) \dx \\
& \qquad + \frac{1}{2}\bigg( \int_{\FT} \abs{w}^2 
\dx + \int_{\FO} \abs{w}^2 \dx \bigg)
+ \int_{\FD} \alpha \abs{w}^2 \dS \\
& \qquad + \int_{\Ftime} \left(\alpha 
\abs{\jump{w}_{\bN}}^2 + \beta 
\abs{\jump{\nabla w}_{\bN}}^2\right) \dS \\
& \qquad  + \int_{\Ftime}\im\left(\rule{0mm}{4mm}\mvl{\nabla w} \cdot 
\jump{\conj{w}}_{\bN} - \mvl{w} \jump{\nabla \conj{w}}_{\bN} + 
\jump{w\nabla \conj{w}}_{\bN}\right) \dS \\
 & \stackrel{\cref{EQN::AV-JUMP-IDENTITIES}}
= \Tnorm{w}{\DG}^2. 
\end{align*}
\end{proof}

\begin{proposition}[Continuity]
\label{PROP::CONTINUITY} 
	The sesquilinear form $\cbA{\cdot}{\cdot}$ and the linear functional 
$\ell(\cdot)$ are continuous in the following sense:
\begin{subequations}
\begin{align}
\label{EQN::CONTINUITY-SESQUILINEAR-FORM}
&\abs{\cbA{v}{w}}  \leq 2 \Tnorm{v}{\DGp} \Tnorm{w}{\DG},
& \forall v,w \in \bT(\Th), \\
\label{EQN::CONTINUITY-LINEAR-FUNCTIONAL}
&\abs{\ell(v)}  \leq \Big(2\Norm{\psi_0}{L^2(\FO)}^2 + 2 
\Norm{\alpha^{1/2}\gD}{L^2(\FD)}^2 \Big)^{1/2} \Tnorm{w}{\DGp},
& \forall v \in \bT(\Th).
\end{align}
\end{subequations}
Furthermore, if $\gD = 0$, then  $\abs{\ell(v)} \leq 
\sqrt{2}\Norm{\psi_0}{L^2(\FO)} \Tnorm{w}{\DG}$.
\end{proposition}
\begin{proof}
Let $v, w \in \bT(\Th)$. After applying the triangle and the Cauchy--Schwarz 
inequalities several times we obtain
\begin{align*}
\abs{\cbA{v}{w}} \leq &  \Norm{v^-}{L^2(\Fspa)}\Norm{\jump{w}_t}{L^2(\Fspa)} 
+ \Norm{v}{L^2(\FT)}\Norm{w}{L^2(\FT)} \\
+ & \Big(\Norm{\alpha^{-\frac{1}{2}}\nabla v 
\cdot \vnSpace{\Omega}}{L^2(\FD)} + \Norm{\alpha^{\frac{1}{2}} v}{L^2(\FD)} 
\Big) 
\Norm{\alpha^{\frac{1}{2}}w}{L^2(\FD)}\\
+ & \Big(\Norm{\alpha^{-\frac{1}{2}}\mvl{\nabla v}}{L^2(\Ftime)^d} + 
\Norm{\alpha^{\frac{1}{2}} 
\jump{v}_{\bN}}{L^2(\Ftime)^d}\Big)\Norm{\alpha^{\frac{1}{2}}\jump{w}_{\bN}}{
L^2(\Ftime)^d} \\\def\dx{\mbox{d}\uu{x}}
+ & \Big(\Norm{\beta^{-\frac{1}{2}}\mvl{v}}{L^2(\Ftime)} + 
\Norm{\beta^{\frac{1}{2}}\jump{\nabla 
v}_{\bN}}{L^2(\Ftime)}\Big)\Norm{\beta^{\frac{1}{2}}\jump{\nabla 
w}_{\bN}}{L^2(\Ftime)}.
\end{align*}
Bound \cref{EQN::CONTINUITY-SESQUILINEAR-FORM} is then obtained by 
using Cauchy--Schwarz inequality once again. The bounds for the linear 
operator $\ell(\cdot)$ can be obtained in a similar way.
\end{proof}

The following well-posedness and quasi-optimality theorem is a direct 
consequence of the Lax--Milgram theorem and 
\Cref{PROP::COERCIVITY,PROP::CONTINUITY}.

\begin{theorem}[Quasi-optimality]\label{THEOREM::WELL-POSEDNESS}
For any finite-dimensional subspace $\bVp(\Th)$ of\,  $\bT(\Th)$ there exists 
a unique solution $\uhp \in \bVp(\Th)$ satisfying 
\cref{EQN::VARIATIONAL-TREFFTZ-DG}. Furthermore, the following 
quasi-optimality condition holds:
\begin{equation}
\label{EQN::QUASI-OPTIMALITY}
\Tnorm{\psi - \uhp}{\DG} \leq 3 \inf_{\shp \in \bVp(\Th)}\Tnorm{\psi - 
\shp}{\DGp}.
\end{equation}
\end{theorem}
\Cref{THEOREM::WELL-POSEDNESS} {allows us} to obtain an estimate of the error in the 
mesh-skeleton norm $\Tnorm{\cdot }{\DG}$ by studying the best approximation of 
the exact solution in $\bVp(\Th)$ in the (slightly stronger) 
$\Tnorm{\cdot}{\DGp}$ norm; this is the subject of the next section.

\begin{remark}[Error bounds at time $t_n$]
\Cref{THEOREM::WELL-POSEDNESS} {allows us} to control the $L^2\OO$ norm of the 
Galerkin error at final time $t=T$, but only the $L^2$ norm of the time 
\emph{jump} $\jmp{\psi-\psi\hp}_t$ on each space-like mesh interface 
$\Omega\times\{t_n\}$ (recall the definition of the $\Tnorm{\cdot}\DG$ norm 
\eqref{EQN::DG-NORMS}).
However, a similar bound on the stronger error norm 
$\|\psi^--\psi\hp^-\|_{L^2(\Omega\times\{t_n\})}$, involving only the trace of 
the error from the time-slab $D_n = \Omega\times(t_{n-1},t_n)$, can be 
obtained by 
extending the argument of this section to the partial cylinder 
$\Omega\times(0,t_n)$ (as opposed to the full cylinder $Q$), precisely as in 
\cite[Prop.~1]{Moiola_Perugia_2018}.
\end{remark} 

\begin{remark}[Energy dissipation]\label{Rem:Energy}
It follows from integration by parts that \Cref{EQN::SCHRODINGER-EQUATION} with 
homogeneous boundary conditions 
$\gD = 0$ preserves the energy (or probability) functional $\calE(t; 
\psi):=\frac12\int_{\Omega} |\psi(\bx, t)|^2 \dx$, i.e.\ $\dfhdt{}\calE(t; 
\psi) 
=0$.
Unfortunately, the Trefftz-DG method is not conservative. 
However, as in the case of the acoustic wave equation 
\cite[\S5.3]{Moiola_Perugia_2018}, it is dissipative and the energy loss can 
be quantified in terms of the initial condition error, the jumps of the 
solution on the mesh skeleton and 
the error on $\FD$ due to the weak imposition of the 
boundary conditions, which, as stated in 
\Cref{REMARK::ENERGY-ESTIMATE}, all converge to zero as we refine the 
space--time mesh. More precisely, for $\gD=0$ the Trefftz-DG solution of 
\cref{EQN::VARIATIONAL-TREFFTZ-DG} satisfies
\begin{align*}
\calE(0;\psi_0) - \calE(T; \uhp) = &\
\calE_{loss} : =  \ \delta_{\calE} + \frac{1}{2}\Norm{\psi_0 - 
\uhp}{\FO}^2, \\
\delta_{\calE} : = &\,\Norm{\jump{\uhp}_t}{L^2(\Fspa)}^2 + 
\Norm{\alpha^{1/2} \uhp}{L^2(\FD)}^2  \\
& + \Norm{\alpha^{1/2} \jump{\uhp}_{\bN}}{L^2(\Ftime)^d}^2 + \Norm{\beta^{1/2} 
\jump{\nabla 
\uhp}_{\bN}}{L^2(\Ftime)}^2.
\end{align*}
This follows from the definition of the $\Tnorm{\cdot}\DG$ norm of the solution 
$\psi_h$, the coercivity of the sesquilinear form $\cbA{\cdot\,}{\!\cdot}$, the 
definition of $\ell(\cdot)$ and simple manipulations:
\begin{align*}
\calE(0;\uhp) &+ \calE(T; \uhp) + \delta_{\calE}
\overset{\eqref{EQN::DG-NORMS}}=\Tnorm{\uhp}\DG^2
\overset{\eqref{EQN::COERCIVITY}}=\im\left(\cbA{\uhp}{\psi_h}\right)
\overset{\eqref{EQN::VARIATIONAL-TREFFTZ-DG}}=\im\left(\ell(\uhp)\right)\\
&\overset{(\gD=0)} = \re\int_\FO\psi_0\conj{\uhp}\di \bx
= \calE(0;\psi_0)+\calE(0;\uhp) - 
\frac{1}{2} \Norm{\psi_0 - \uhp}{\FO}^2.
\end{align*}
Further manipulations give the identity 
$\calE_{loss}=\Tnorm{\psi-\uhp}\DG^2-\calE(T,\psi-\uhp)$.
\end{remark}

\begin{remark}
We briefly explain how to modify the method to allow {for} other boundary conditions. 
Let $\deO\times I$ be partitioned in three parts with Dirichlet, Neumann and 
Robin conditions:
\begin{align*}
\psi& =\gD  \quad \oon\ \FD=\GD\times I,\\
\nabla\psi\cdot\bn^x_\Omega & = \gN \quad \oon\ \FN=\GN\times I,\\
\nabla\psi\cdot\bn^x_\Omega - i \vartheta\psi & = g_{\mathrm R}\quad  \oon\ 
\FR=\GR\times I,
\end{align*}
for some positive ``impedance'' function $\vartheta \in L^\infty(\FR)$.
We extend the numerical flux parameter $\beta$ to $\Ftime\cup\FN$, fix a 
function $\delta$ on $\FR$ such that $0 < \delta <1$ and choose the 
numerical fluxes on $\deO\times I$ as
\begin{align*}
\label{EQN::NUMERICAL-FLUXES-MIXED-BC}
	\hUhp & := \left\{
	\begin{tabular}{ll}
	$\gD$, & on $\FD$, \\[1ex]
	$\uhp - i\beta \left(\nabla \uhp \cdot \vnSpace{\Omega} - \gN\right)$, & on 
$\FN$, \\[1ex]
	$\uhp + \delta(i \vartheta )^{-1} \left(\nabla \uhp \cdot \vnSpace{\Omega} - 
i\vartheta  \uhp - g_{\mathrm{R}}\right) $, & on $\FR$,
	\end{tabular}
	\right. \\[1mm]
	\QFLUX \cdot \vnSpace{\Omega} & := \left\{
	\begin{tabular}{ll}
	$\nabla \uhp \cdot \vnSpace{\Omega}  + i \alpha \left(\uhp - \gD 
\right)$, & on $\FD$,\\[1ex]
	$\gN$, & on $\FN$,\\[1ex]
	$ \nabla \uhp \cdot \vnSpace{\Omega} - (1 - \delta) \left(\nabla \uhp 
\cdot \vnSpace{\Omega} - i\vartheta  \uhp  - g_{\mathrm{R}}\right) $, & 
on $ \FR$.
	\end{tabular}
	\right.
\end{align*}
This amounts to adding the following boundary terms to $\cbA{\uhp}{\shp}$ and 
$\ell(\shp)$ 
in \eqref{EQN::VARIATIONAL-TREFFTZ-DG}, respectively:
\begin{align*}
\bullet&\int_{\FN} \big(-\uhp \nabla \Conjugate{\shp} \cdot \vnSpace{\Omega} + 
i\beta \left(\nabla \uhp \cdot \vnSpace{\Omega}\right)
(\nabla \Conjugate{\shp} \cdot \vnSpace{\Omega})\big) \dS \\
&\qquad+ \int_{\FR} \Big(\delta\nabla\uhp\cdot\vnSpace{\Omega}
+(1-\delta)i\vartheta \uhp\Big)\Big(\cs+\frac{i}\vartheta 
\nabla\cs\cdot\vnSpace{\Omega}\Big)\dS,
\\
\bullet & \int_{\FN} \gN \left(-\Conjugate{\shp} + i\beta \nabla 
\Conjugate{\shp} 
\cdot \vnSpace{\Omega}\right) \dS
+\int_{\FR} g_{\mathrm R}\Big((\delta - 1) \Conjugate{\shp} + 
\frac{i\delta}{\vartheta} \nabla \Conjugate{\shp} \cdot \vnSpace{\Omega} \Big)\dS .
\end{align*}
In \cref{EQN::DG-NORMS} both norms $\Tnorm{\uhp}{\DG}^2$ and 
$\Tnorm{\uhp}{\DGp}^2$ have to be supplemented by the terms
\begin{align*}
\Norm{\beta^{1/2} \nabla \uhp \cdot \vnSpace{\Omega}}{L_2(\FN)}^2 \!+ 
\Norm{\big(\vartheta (1 - \delta)\big)^{1/2} \uhp}{L_2(\FR)}^2 + \Norm{\left(\delta 
\vartheta^{-1}\right)^{1/2} \nabla \uhp \cdot \vnSpace{\Omega}}{L_2(\FR)}^2,
\end{align*}
and a further term $\Norm{\beta^{-1/2} \uhp}{L_2(\FN)}^2$ has to be 
added to $\Tnorm{\uhp}{\DGp}^2$. 
If $\FR\ne\emptyset$, the continuity constant $2$ of the sesquilinear form 
$\cbA{\cdot}{\cdot}$ in \Cref{PROP::CONTINUITY} becomes
\[
2 \max \bigg\{\Norm{\Big(\frac{1 - \delta}{\delta}\Big)^{1/2}}{L_\infty(\FR)}, \ 
\Norm{\Big(\frac{\delta}{1 - \delta}\Big)^{1/2}}{L_\infty(\FR)} \bigg\}.
\]
Then the a priori analysis carried out in this section can be extended to 
this setting.
\end{remark}

		
\section{Approximation and convergence analysis of the Trefftz-DG 
method\label{SECT::ERROR-ESTIMATE}}

In this section we prove $h$-estimates on the $\Tnorm{\cdot}{\DGp}$ norm of 
the best 
approximation error of Schr\"odinger solutions by discrete Trefftz functions in 
$\bVp(\Th)=\prod_{K\in\calT_h(Q)}\IT_p(K)$. The local space $\bVp(K)$ is defined 
for each $K = \Kx \times I_n \in \Th(Q)$ and for $p\in\IN$ as the following set 
of complex exponentials:
\begin{align}
\label{EQN::BASIS-FUNCTIONS}
\bVp(K)  :=& \spn\big\{ \phi_{\ell}(\bx, t),\ \ell=1,\ldots, n_{d,p}\big\}, 
\ \mbox{ where }\\
\phi_{\ell}(\bx, t) := &
\exp\left[i\left(k_\ell \bd_\ell^\top \bx - (k_{\ell}^2 + V|_{K})t\right)\right] 
\ \mbox{ for } \ell = 1, \ldots, 
n_{d,p},\nonumber 
\end{align}
for some parameters $\left\{k_{\ell}\right\}\subset \IR$ and
directions $\left\{\bd_\ell\right\}\subset 
\cS_1^d:=\{\bv\in\IR^d,|\bd|=1\}$, which can be chosen differently in each 
cell $K$.
It is immediate to verify that $i\derp{\phi_\ell}t+\Delta\phi_\ell-V\phi_\ell=0$ 
in~$K$.
These exponential solutions are called ``pseudoplane waves'' in the context of 
Fresnel optics in \cite[eq.~(11)]{Grella_1982}.

In the following we choose the local dimension $n_{d,p}$ in dependence of
the dimension $d$ of the problem and  a ``degree'' parameter $p$.
The parameter $p$ is to be understood as the degree of the polynomial space 
$\IP^p(K)$ providing the same $h$-convergence rates of $\IT_p(K)$.
The Trefftz property {allows us} to construct $\IT_p(K)$ with dimension $n_{d,p}\ll 
\DIM{\IP^p(K)}$.

The approximation theory in \cref{SECT::GENERAL-APPROXIMATION} is not restricted 
to the complex-exponential basis functions in \eqref{EQN::BASIS-FUNCTIONS}, 
in fact it {allows for} any Trefftz basis satisfying \cref{ConditionApprox}.
This condition requires that the discrete space can ``replicate'' the Taylor polynomial of degree $p$ of any Schr\"odinger solution.
We expect that the same condition can be shown for other Trefftz bases that can be computationally convenient for some geometries $\Omega$ or potentials $V$.
For instance, in two space dimensions, in the presence of circular symmetries, one could define a discrete Trefftz space with basis
$b_{\ell,m}(r,\theta,t)=J_\ell(k_m r)\ee^{\ri\ell\theta}\ee^{-\ri k_m^2t}$ for a set of $\ell\in\IZ$ and $k_m\in\IR$, where $J_\ell(\cdot)$ is the $\ell$th first-order Bessel function and $(r,\theta)$ are the polar coordinates in the plane.

\subsection{Notation and preliminary results\label{SUBSECTION::PRELIMINARS}}
We use the standard multi-index notation for partial 
derivatives and monomials, adapted to the space--time setting: for $\bj = 
(\bjx, j_t) = \left(j_{x_1}, \ldots, j_{x_d}, j_t\right)\in \IN_0^{d+1},$
\begin{align*}
\bj!&:= j_{x_1}!\cdots j_{x_d}!j_t!,
&\abs{\bj} &:= \abs{\bjx} + j_t := j_{x_1} + \cdots + j_{x_d} + j_t,\\
\DerA{\bj}{f} &:= \PartDer{j_{x_1}}{x_1}\cdots 
\PartDer{j_{x_d}}{x_d}\PartDer{j_t}{t} f,
&\bx^{\bjx}t^{j_t} &:= x_1^{j_{x_1}}\cdots x_d^{j_{x_d}}t^{j_t}.
\end{align*}

\begin{definition}[Taylor polynomial]
\label{DEF::TAYLOR-POLYNOMIAL}
On an open and bounded set $D\subset\IR^{d+1}$,
the Taylor polynomial of order $m\in\IN$ (and degree $m - 1$), centered at $(\bz,s)\in D$,
of a function $\varphi \in \EFC{m - 1}{D}$ is denoted
\begin{equation*}
\Taylor{(\bz,s)}{m}{\varphi}(\bx, t) := \sum_{\abs{\bj} < m} 
\frac{1}{\bj!} \DerA{\bj}{\varphi}(\bz, s)(\bx - \bz)^{\bjx} (t - s)^{j_t}. 
\end{equation*}
\end{definition}
If $\varphi \in \EFC{m}{D}$ and the segment 
$[(\bz,s),(\bx,t)]\subset D$, 
the Lagrange's form of the Taylor remainder (see \cite{Callahan_2010}, Corollary 
3.19) is bounded as follows:
\begin{align}
\nonumber
\abs{\varphi(\bx,t) - \Taylor{(\bz,s)}{m}{\varphi}(\bx,t)} 
& \leq  \abs{\varphi}_{{\EFC{m}{D}}} \sum_{\abs{\bj} = m} \frac{1}{\bj!} {\left|\left(\bx - 
\bz\right)^{\bjx} (t - s)^{j_t}\right|}\\
\label{EQN::TAYLOR-REMAINDER}
& \leq \frac{(d + 1)^{m/2}}{m!}
h_D^{m} 
\abs{\varphi}_{\EFC{m}{D}},
\end{align}
where $h_D$ is the diameter of $D$, and we used the multinomial theorem 
($(\sum_{r=1}^{d+1}v_r)^m=\sum_{|\bj|=m}\frac{m!}{\bj!}
\bv^\bj$) and 
$|\bv|_1\le (d+1)^{1/2}|\bv|_2 \; \forall\bv\in\IC^{d+1}$.

In order to prove approximation results for solutions belonging to spaces more 
general than $\EFC{m}{Q}$, we 
introduce the averaged Taylor polynomial as presented in \cite{Duran_1983}
(see also a refined version in \cite[Def.~4.1.3]{Brenner_Scott_2007}). 
\begin{definition}[Averaged Taylor polynomial]
\label{DEF::AVERAGED-TAYLOR-POLYNOMIAL}
Let $D \subset \IR^{d + 1}$, $1 \leq 
d \in \IN$, be an open and bounded set, with diameter $h_D$, star-shaped with 
respect to the ball $B:= B_{\rho h_D}(\bz, s)$ centered at $(\bz, s) \in D$ and 
with radius $\rho h_D$, for some $0 < \rho \leq \frac{1}{2}$. If $\varphi \in 
\ESOBOLEV{m-1}{D}$, the \emph{averaged Taylor polynomial} of order $m$ (and degree $m 
- 1$) is defined as
\begin{equation}
\label{EQN::AVERAGED-TAYLOR-POLYNOMIAL}
\AvTaylor{m}{\varphi}(\bx, t) := 
\frac{1}{\abs{B}} \int_{B} \Taylor{(\bz,s)}{m}{\varphi}(\bx, t) \dV(\bz,s).
\end{equation}
\end{definition}

We recall the result of \cite[Prop.~(4.1.17)]{Brenner_Scott_2007}:
\begin{equation}\label{EQN::DQ}
D^\bj\Taylor{(\bz,s)}{m}{\varphi} = 
\Taylor{(\bz,s)}{m-|\bj|}{D^\bj\varphi},\qquad 
D^\bj\AvTaylor{m}{\varphi} = \AvTaylor{m-|\bj|}{D^\bj\varphi},\qquad
|\mj|<m.
\end{equation}
The following Bramble--Hilbert lemma provides an estimate for the error of the 
averaged Taylor polynomial, see \cite{Duran_1983} 
and \cite[Thm.~4.3.8]{Brenner_Scott_2007}.
\begin{lemma}[Bramble--Hilbert]
\label{LEMMA::Bramble--HILBERT}
Under the conditions in \Cref{DEF::AVERAGED-TAYLOR-POLYNOMIAL}, if 
$\varphi \in \ESOBOLEV{m}{D}$, then
\begin{equation}
\label{EQN::Bramble--HILBERT-LEMMA}
\Norm{\DerA{\bj}{\left(\varphi - \AvTaylor{m}{\varphi}\right)}}{L^2(D)} \leq 
C_{d,m, \rho} \:
h_D^{m-\abs{\bj}} \abs{\varphi}_{\ESOBOLEV{m}{D}}, \quad \mbox{ for } \abs{\bj} 
\leq m.
\end{equation}
\end{lemma}
The explicit value of $C_{d,m,\rho}>0$ is stated in 
\cite[p.~986]{Duran_1983} in dependence of $d$, $m$ and $\rho$, with the 
last of these parameters as in \Cref{DEF::AVERAGED-TAYLOR-POLYNOMIAL}.
The last key ingredient is the trace inequality, 
\cite[Theorem 1.6.6]{Brenner_Scott_2007}, which
can be written
for any element $K = \Kx \times I_n \in \Th (Q)$
in our space--time setting as
\begin{align}
\Norm{\varphi}{L^2(\Kx \times \left\{t_{n-1}, t_n\right\})}^2 
\leq 
C_{\tr}\left( h_n^{-1} \Norm{\varphi}{L^2(K)}^2 + h_n 
\Norm{\partial_t \varphi}{L^2(K)} ^2\right), 
\  \forall \varphi \in \ESOBOLEV{1}{I_n; L^2(\Kx)}, 
\label{EQN::TRACE-INEQUALITY-TIME}
\\
\Norm{\varphi}{L^2(\partial \Kx \times I_n)}^2 
\leq C_{\tr} \left(h_{\Kx}^{-1} \Norm{\varphi}{L^2(K)}^2 + h_{\Kx} 
\Norm{\nabla \varphi}{L^2(K)^d}^2 \right), 
\  \forall \varphi \in L^2\left(I_n; \ESOBOLEV{1}{\Kx}\right),
\nonumber
\end{align}
where
$C_{\tr}$ only depends on the shape-regularity parameter $\mathsf{sr}(\Th)$ of 
the space mesh.


\subsection{General approximation estimate} \label{SECT::GENERAL-APPROXIMATION}
We now give a condition on the discrete Trefftz space $\bVp(K)$ and show how it 
entails $h$-approximation estimates.
In the next sections we describe concrete spaces that satisfy this condition.

\Cref{ConditionApprox} states that for any sufficiently smooth Schr\"odinger solution $\psi${, for each element $K\in \Th$, and almost every point $(\bz,s)$ in a ball $B\subset K$, the discrete Trefftz space contains an element whose Taylor polynomial centered at $(\bz,s)$ matches that of $\psi$.}
We also require {the coefficients $(a_1,\ldots,a_{n_{d,p}})$ of this approximant to be uniformly bounded with respect to the points $(\bz,s)$ in the aforementioned ball $B$}, so this is a condition on the basis rather than on the discrete 
space itself.
To allow {for} general $\psi\in H^{p+1}(K)$, whose Taylor polynomial might not be 
defined everywhere and does not guarantee enough approximation, we match 
$\Taylor{(\bz,s)}{p+1}{\psi}$ in almost every point in a ball $B$ (see 
\Cref{rem:PointDerivatives,rem:HvsC}).

\begin{condition}\label{ConditionApprox}
Let $B\subset K$ be a $(d+1)$-dimensional ball such that $K$ is star-shaped 
with respect to $B$.
Let $\{\phi_1,\ldots,\phi_{n_{d,p}}\}\subset \calC^\infty(K)$ be a basis of 
$\IT^p(K)$.
For every $\psi\in\bT(K)\cap H^{p+1}(K)$, there exists a vector-valued 
function $\ba\in L^1(B)^{n_{d,p}}$ satisfying the following two conditions
\begin{subequations}
\begin{align}
D^\bj\psi(\bz,s)&=\sum_{\ell=1}^{n_{d,p}} a_\ell(\bz,s) D^\bj\phi_\ell(\bz,s) 
\quad \text{for all } |\bj|\le p
\quad\text{and a.e. }(\bz,s)\in B,
\label{EQN::Condition1} 
\\
\N{|\ba|_1}_{L^1(B)}&=\int_{B}\sum_{\ell=1}^{n_{d,p}} 
|a_\ell(\bz,s)| \dV(\bz,s)
\le C_\star |K|^{1/2}\N{\psi}_{H^{p+1}(K)},
\label{EQN::Condition2}
\end{align}
\end{subequations}
where $C_\star>0$ might depend on $d,p,$ {and }$ \{\phi_\ell\}$ but is independent of $K$ and $\psi$.
\end{condition}

Multiplying \cref{EQN::Condition1} with 
$\frac{(\bx-\bz)^{\bj_x}(t-s)^{j_t}}{\bj!}$ and summing over $\bj$, we observe 
that \cref{ConditionApprox} implies that for every smooth Schr\"odinger 
solution $\psi$ there is an element of the discrete space $\bVp(K)$ whose Taylor 
polynomial at $(\bz,s)$ coincides with that of $\psi$:
\begin{align}
\nonumber
\Taylor{(\bz, s)}{p + 1}{\psi}(\bx,t) 
\ = \ \Taylor{(\bz, s)}{p + 1}{\sum_{\ell = 1}^{n_{d,p}} a_\ell(\bz, s) \phi_{\ell}}(\bx,t)  
\ = \ \sum_{\ell = 1}^{n_{d,p}} a_\ell(\bz, s) \Taylor{(\bz, s)}{p + 1}{\phi_{\ell}}(\bx,t) \\
\text{for a.e. }(\bz,s)\in B,\quad \forall (\bx,t)\in K. 
\label{EQN::T=T}
\end{align}

\begin{remark}\label{rem:PointDerivatives}
\Cref{ConditionApprox} requires the point value of the partial derivatives of a 
Sobolev function in almost every point of $B$.
This is to be understood
as follows.

Let $\psi\in H^{p+1}(K)$.
By Calder\'on's extension theorem \cite[Thm.~A.4]{McLean_2000}, $\psi$ can be extended to a $\widetilde\psi\in H^{p+1}(\IR^{d+1})$. 
{This allows us to apply \cite[Thm.~10.1.4]{Adams_Hedberg_1996}, which} implies that there exists a zero-measure set $\Upsilon\subset\IR^{d+1}$ such that $\widetilde\psi$ admits a ``differential of order $p$'' at every point $(\bz,s)$ of the complement of $\Upsilon$.
The differential of order $p$, as defined in \cite[Def.~10.1.3]{Adams_Hedberg_1996}, is a polynomial that coincides with the Taylor polynomial $T^{p+1}_{(\bz,s)}[\widetilde\psi]$ whenever this is defined.
With a small abuse of notation, this is the polynomial appearing {on} the left-hand side of \cref{EQN::T=T} for all $(\bz,s)\in B\setminus\Upsilon$ (while $\phi_\ell\in \calC^\infty(K)$ so their Taylor polynomials are the classical ones).
The partial derivatives $D^\bj\psi(\bz,s)$ in \eqref{EQN::Condition1} are the 
partial derivatives of $T^{p+1}_{(\bz,s)}[\psi]$ evaluated in $(\bz,s)$.
\end{remark}

\begin{proposition}
\label{PROP::H-ESTIMATE}

Let $K\in \Th(Q)$ and $\psi \in \bT(K) \cap \ESOBOLEV{p+1}{K}$, with $K$ 
satisfying 
the conditions in \Cref{DEF::AVERAGED-TAYLOR-POLYNOMIAL}. 
Assume that $\bVp(K)$ satisfies \Cref{ConditionApprox}.
Then, there exists an element $\Phi \in \bVp(K)$ and $C>0${, which depends on the constant $C_\star$ in \eqref{EQN::Condition2} but is} independent of $h_K$ and $\psi$, such that:
\begin{align}
\label{EQN::LOCAL-ESTIMATE}
\Norm{\DerA{\bj}{\left(\psi - \Phi\right)}}{L^2(K)} \leq C h_{K}^{p + 1 - 
|\bj|} 
\Norm{\psi}{\ESOBOLEV{p+1}{K}}, \qquad 0 \leq \abs{\bj} \leq p.
\end{align}
\end{proposition}

\begin{proof}
Let $\ba\in L^1(B)$ be the coefficient function defined 
by \Cref{ConditionApprox}.
We define the discrete Trefftz function $\Phi\in \bVp(K)$ as
\begin{align}
\nonumber
\Phi(\bx, t) & := \frac{1}{\abs{B}}\int_B \sum_{\ell = 1}^{n_{d,p}} 
a_\ell(\bz, s) \phi_\ell(\bx, t) \dV(\bz,s)  \\
 & = \frac{1}{\abs{B}} 
\sum_{\ell = 1}^{n_{d,p}}\left(\int_B a_\ell(\bz, s) \dV(\bz, s)\right) 
\phi_\ell(\bx,t). \label{EQN::Phi}
\end{align}
Since $\Phi \in\EFC{\infty}{\overline{K}}$, by the triangle inequality, for
$|\bj|\le p$ we have 
\begin{equation}\label{EQN::TRIA}
\Norm{\DerA{\bj}{\left(\psi - \Phi\right)}}{L^2(K)} \leq  
\Norm{\DerA{\bj}{\left(\psi - 
\AvTaylor{p+1}{\psi}\right)}}{L^2(K)}   + 
\Norm{\DerA{\bj}{\left(\AvTaylor{p+1}{\psi} - \Phi\right)}}{L^2(K)},
\end{equation}
{where $\AvTaylor{p+1}{\cdot}$ is the averaged Taylor polynomial  of order $p + 1$ defined in \cref{EQN::AVERAGED-TAYLOR-POLYNOMIAL}.}

The first term is bounded by the Bramble--Hilbert \cref{LEMMA::Bramble--HILBERT}, while for the second term we first derive a 
pointwise estimate for all $(\bx, t) \in K$ and $\abs{\bj} \leq p$ using the 
multivariate Taylor's theorem and the fact that $\phi_{\ell} \in 
\EFC{\infty}{K}$:
\begin{align*}
\big| D^{\bj} & \big( \AvTaylor{p+1}{\psi} (\bx,t)  - \Phi(\bx,t)\big)\big| \\
&\overset{\cref{EQN::AVERAGED-TAYLOR-POLYNOMIAL}, \cref{EQN::Phi}}= 
\abs{\frac1{|B|}\int_B D^\bj\bigg(\Taylor{(\bz,s)}{p+1}{\psi} (\bx,t)
-\sum_{\ell=1}^{n_{d,p}}a_\ell(\bz,s)\phi_\ell(\bx,t)\bigg)\dV(\bz,s)}
\\
&\quad\overset{\cref{EQN::T=T}}=
\abs{\frac1{|B|}\int_B \sum_{\ell=1}^{n_{d,p}} a_\ell(\bz,s)
D^\bj\bigg(\Taylor{(\bz,s)}{p+1}{\phi_\ell} (\bx,t)-\phi_\ell(\bx,t)\bigg)\dV(\bz,s)}
\\
& \quad \overset{\cref{EQN::DQ}}= 
\abs{\frac1{|B|}\int_B \sum_{\ell=1}^{n_{d,p}} a_\ell(\bz,s)
\bigg(\Taylor{(\bz,s)}{p+1-|\bj|}{D^\bj\phi_\ell} (\bx,t)-D^\bj\phi_\ell(\bx,t)\bigg)\dV(\bz,s)}
\\
& \quad \overset{\cref{EQN::TAYLOR-REMAINDER}}\leq 
\frac{(d+1)^{\left(p + 1 - \abs{\bj}\right)/2}}
{\abs{B}(p+1-\abs{\bj})!} h_K^{p+1-\abs{\bj}}
\max_{\ell = 1, \ldots, n_{d,p}} 
\left\{\abs{\phi_{\ell}}_{\EFC{p+1}{K}}\right\} 
\int_B 
\abs{\ba(\bz, s)}_1 \dV(\bz, s)\\ 
& \quad \overset{\cref{EQN::Condition2}}
\leq  
\frac{|K|^{1/2}(d+1)^{\left(p + 1 - \abs{\bj}\right)/2}}
{|B|(p+1-\abs{\bj})!} h_K^{p+1-\abs{\bj}}
\max_{\ell = 1, \ldots, n_{d,p}} 
\left\{\abs{\phi_{\ell}}_{\EFC{p+1}{K}}\right\} 
C_\star \N{\psi}_{H^{p+1}(K)}.
\end{align*}
Then
\begin{align*}
&\Norm{\DerA{\bj}{\left(\AvTaylor{p+1}{\psi} - \Phi\right)}}{L^2(K)} 
\\&\leq 
\frac{|K|}{|B|}
\frac{(d+1)^{\left(p + 1 - \abs{\bj}\right)/2}}{(p+1-\abs{\bj})!}
h_K^{p + 1-\abs{\bj}}
\max_{\ell = 1, \ldots, n_{d,p}}
\left\{\abs{\phi_{\ell}}_{\EFC{p+1}{K}}\right\}
C_\star\N{\psi}_{H^{p+1}(K)},
\end{align*}
which, combined with $|K|\le |B_{h_K}|\le \rho^{d+1}|B_{\rho 
h_K}|=\rho^{d+1}|B|$, the triangle inequality \cref{EQN::TRIA} and 
Bramble--Hilbert lemma estimate \cref{EQN::Bramble--HILBERT-LEMMA}, completes 
the proof.
\end{proof}

No special property of the Schr\"odinger equation has been used in the proof 
of \cref{PROP::H-ESTIMATE}:
the same result extends to any linear PDE for which one has at hand a discrete Trefftz space that can reproduce the Taylor polynomials of any PDE solution.

\begin{remark}\label{rem:HvsC}
The use of averaged Taylor polynomials $\AvTaylor{m}{\cdot}$ 
\cref{EQN::AVERAGED-TAYLOR-POLYNOMIAL} {allows us} to treat any $\psi\in H^{p+1}(K)$.
Under the stronger assumption $\psi\in C^{p+1}(K)$ the argument of  
\Cref{PROP::H-ESTIMATE} simplifies considerably:
one could use the standard Taylor polynomial $\Taylor{(\bz,s)}{m}{\cdot}$,
assume that $K$ is star-shaped with respect to a point,
and require the identity \cref{EQN::Condition1} for a single point $(\bz,s)$ 
only.
\end{remark}


\subsection{Best approximation in the (1+1)-dimensional case}
\label{SECT::APPROX-1D}

The next proposition shows that, in one space dimension, the choice of 
exponential basis functions \cref{EQN::BASIS-FUNCTIONS} with different values of 
$k_\ell$ (and equal direction $d_\ell=1$) is enough to ensure 
\cref{ConditionApprox} and the approximation properties of $\IT^p(K)$.
Only $2p+1$ degrees of freedom per element are needed to obtain $h^p$ 
convergence.
An example of the basis \eqref{EQN::BASIS-1+1} for $p = 3$ is plotted in 
\cref{fig:1Dbasis}.

\begin{proposition}
\label{PROP::POINTWISE-TAYLOR-1D}
Let $d=1$, $p\in\IN$, $n_{1,p}=2p+1$ and the parameters 
$\{k_\ell\}_{\ell=1}^{2p+1}\subset\IR$ be all different from one another.
Let
\begin{equation}\label{EQN::BASIS-1+1}
\phi_{\ell} (x,t)=\exp\big[i\big(k_\ell x - (k_\ell^2 + V|_K)t\big)\big],
\qquad \ell=1,\ldots,2p+1,
\end{equation}
be the basis of the discrete Trefftz space $\IT^p(K)$.
Then \cref{ConditionApprox} is satisfied.
\end{proposition}
 
\begin{figure}[htb]
\includegraphics[width=\textwidth,clip,trim=80 20 70 20]{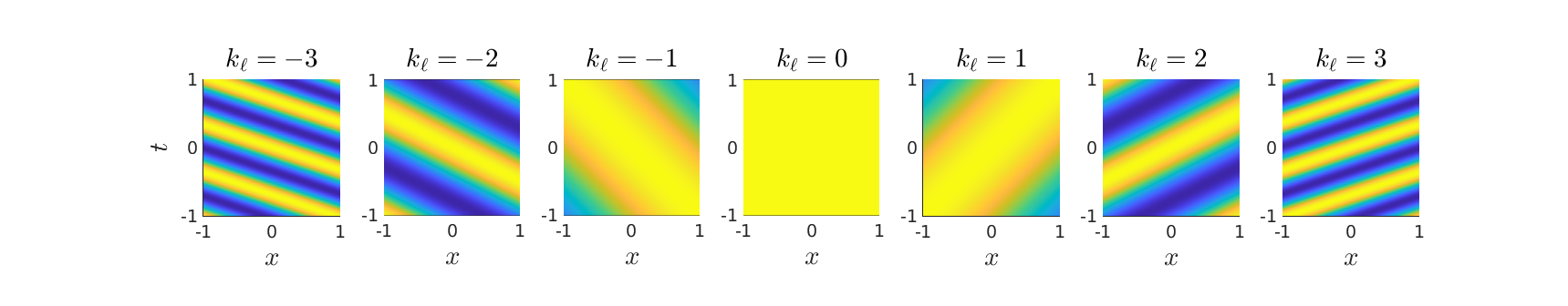}
\caption{The real parts ($\cos(k_\ell x-k_\ell^2t)$) of the Trefftz basis 
functions $\phi_\ell$ of \eqref{EQN::BASIS-1+1} for the potential $V=0$ and 
$p=3$, plotted on the space--time square $(x,t)\in(-1,1)\times(-1,1)$.
The parameters $k_\ell$ are chosen as $k_\ell=-3,-2,\ldots,3$.
These functions can be thought as plane waves $\phi_\ell(x,t)=\ee^{\ri \kappa (v_x,v_t)^\top(x,t)}$ in the $xt$ plane with different wavenumbers $\kappa=\sqrt{k_\ell^2+k_\ell^4}$ and directions $(v_x,v_t)=(k_\ell, -k_\ell^2)/\sqrt{k_\ell^2+k_\ell^4}$.
The space- and time-frequencies increase linearly and quadratically with 
$|k_\ell|$, respectively.
The space spanned by these $n_{1,3}=7$ basis functions approximates Schr\"odinger solutions with the same convergence rates of cubic polynomials.
}
\label{fig:1Dbasis}
\end{figure}

\begin{proof}
Let $\psi\in\bT(K)\cap H^{p+1}(K)$ and let $\Upsilon$ be the zero-measure set 
introduced in \cref{rem:PointDerivatives}.
Let $(z,s)\in B\setminus \Upsilon$.
Since all the derivatives of $\psi$ of order at 
most  $p$ are defined at $(z, s)$ and each basis function 
$\phi_\ell \in \EFC{\infty}{K}$, the Taylor polynomials centered at $(z, s)$ of 
order $p + 1$ of $\psi$ and each $\phi_\ell$ can be written as:
\begin{subequations}
\begin{align}
\label{EQN::ATP-PSI} 
\Taylor{(z,s)}{p+1}{\psi}(x, t) & = \sum_{j_x=0}^p \sum_{j_t = 0}^{p - j_x} 
b_\bj(x - z)^{j_x}(t - s)^{j_t}, \qquad \mbox{ where } \bj = (j_x, j_t),\\
\label{EQN::ATP-PHI}
\Taylor{(z,s)}{p+1}{\phi_{\ell}}(x, t) & = \sum_{j_x=0}^p \sum_{j_t = 0}^{p - 
j_x} M_{\bj, \ell}(x - z)^{j_x}(t - s)^{j_t} \quad \mbox{ for }\ell = 
1,\ldots, n_{1,p},
\end{align}
\end{subequations}
{where $\left\{b_\bj\right\}$ and $\left\{M_{\bj, \ell}\right\}$ denote the Taylor polynomial coefficients for $\psi$ and $\phi_{\ell}$, respectively, as in \Cref{DEF::TAYLOR-POLYNOMIAL}}.
We aim to prove that there exists $\ba{(z,s)} \in \IC^{2p+1}$ that 
satisfies 
\cref{EQN::Condition1}, i.e., 
$$\sum_{\ell = 1}^{n_{1,p}} a_{\ell}(z,s) M_{\bj, \ell} = b_{\bj},
\quad 
\mbox{for } |\bj| \leq p, \quad \bj = (j_x, j_t).
$$
This can be arranged as a rectangular linear system of the form 
$\bM \,\ba(z,s) = \bb$, with $\bb \in \IC^{r_p}$ and $\bM \in \IC^{r_p \times 
n_{1,p}}$, $r_p = (p+1)(p+2)/2{\ge n_{1,p}}$. 

The case $p = 1$ is rather special, since $n_{1,1} = r_1 = 3$ and
the matrix $\bM$ is square; hence we just need to show that it is not singular. 
For $p \geq 2$, {in order to prove the existence of such an $\ba(z,s)$, we need to show that $\bb\in\IMG\bM$.
To this purpose, after equation \eqref{EQN::TAYLOR-COEFFICIENTS-RELATION} below, we define a set $\calD\subset\IC^{r_p}$ that contains $\{\bb\}\cup\IMG\bM$, and the rest of the proof is devoted to show that $\dim\calD=\RANK\bM$.}

We first note that for all smooth 
$\varphi \in \bT(K)$ and each multi-index 
$\bj$, with $\abs{\bj} \leq p$, $\DerA{\bj}{\varphi}(z,s) = 
\DerA{\bj}{\Taylor{(z,s)}{p + 1}{\varphi}}(z,s)$.
Since $\varphi$ and its derivatives satisfy the Schr\"odinger equation, this 
implies that $D^\bj\Taylor{(z,s)}{p + 1}{\varphi}$ satisfies the same equation 
at the single point $(z,s)$:
\begin{equation}\label{eq:SchrDT}
\bigg(\Big(i\derp{}t+\Delta-V\Big)D^\bj\Taylor{(z,s)}{p + 
1}{\varphi}\bigg)(z,s)=0, 
\quad  \tfor p \geq 2 \; \tand \; |\bj|\le p-2. 
\end{equation}
We want to express these equations in terms of the coefficients of 
$\Taylor{(z,s)}{p+1}{\varphi}(x,t)$.
Setting $\bC\in\IC^{r_p}$ to be the vector with components $C_{j_x\, j_t}:=\frac1{\bj!}D^\bj\varphi(z,s)$ {with $\bj=(j_x,j_t)$}, so that 
$\Taylor{z,s}{p+1}{\varphi}(x,t)=\sum_{|\bj|\le p} 
C_{j_x\,j_t}(x-z)^{j_x}(t-s)^{j_t}$, we have
\begin{align*}
\dfdt{}\Taylor{z,s}{p+1}{\varphi}(x,t) & = \sum_{j_x = 0}^p 
\sum_{j_t=0}^{p-1-j_x} C_{j_x (j_t+1)}(j_t+1)(x- z)^{j_x}(t-s)^{j_t}, \\
\frac{\partial^2}{\partial x^2}\Taylor{z,s}{p+1}{\varphi}(x,t) & 
= \sum_{j_x = 0}^{p - 2} \sum_{j_t=0}^{p-2-j_x} 
C_{(j_x + 2) {j_t}}(j_x+1)(j_x+2)(x- z)^{j_x} (t-s)^{j_t}.
\end{align*}
Expanding \eqref{eq:SchrDT}, we get the following $p(p-1)/2$ relations between 
the coefficients of the 
Taylor polynomial of $\varphi$:
\begin{align}
\label{EQN::TAYLOR-COEFFICIENTS-RELATION}
i(j_t+1)C_{j_x(j_t+1)} + (j_x+1)(j_x+2)C_{(j_x+2)j_t}= V C_{j_x,j_t},
\\ \tfor |\bj| \le p - 2, \quad \bj = (j_x, j_t). \nonumber
\end{align}
We define $\calD : = \left\{\bC \in \IC^{r_p}\ \big| \ \bC \mbox{ 
satisfies \cref{EQN::TAYLOR-COEFFICIENTS-RELATION}}\right\}$: since 
$\psi,\phi_\ell\in\bT(K)$, it is evident 
that $\bb \in \calD$ and $\IMG{\bM} \subset \calD$.

The key point of the proof is to show that 
$\RANK{\bM} = \mbox{dim}(\calD) = 2p + 1$, which guarantees that the 
overdetermined linear system 
$\bM \ba(z,s) = \bb$ has a  unique solution.

\Cref{FIG::INDEPENDENT-RELATIONS} illustrates the relations that define $\calD$ 
for both cases $V = 0$ and $V \neq 0$.
It shows that given the $2p+1$ entries $C_{j_x j_t}$ of $\bC\in \calD$ 
corresponding to $j_x {\in\{ 0, 1\}}$, all other entries are uniquely 
determined by the conditions \cref{EQN::TAYLOR-COEFFICIENTS-RELATION}.
Therefore
\[\RANK{\bM} \leq \DIM{\calD} \leq 2p + 1.\]

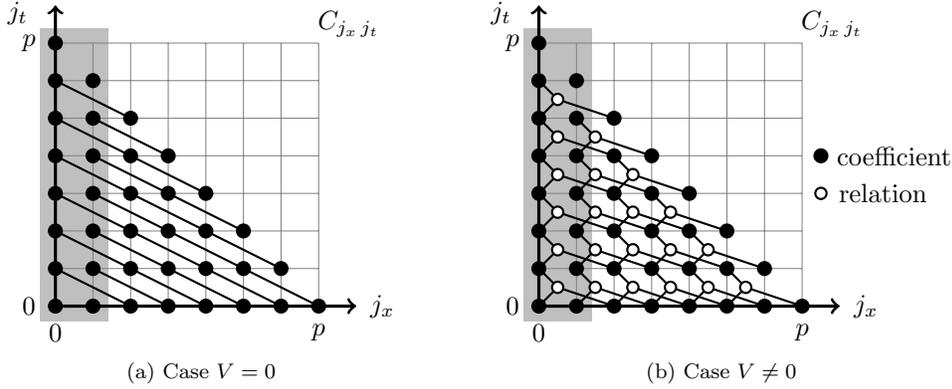
\begin{figure}[htb]
\centering
\subfloat[Case $V = 0$ ]{
\begin{tikzpicture}[scale=.5]
\fill[lightgray](-.4,-.4)--(1.4,-.4)--(1.4,7.4)--(-.4,7.4);
\draw[step=1cm,gray,very thin] (0,0) grid (7,7);
\draw[very thick,->](0,0)--(8,0);\draw[very thick,->](0,0)--(0,8);
\draw(8.7,0)node{$j_x$};\draw(-1,7.8)node{$j_t$};
\draw(-.7,7)node{$p$};\draw(-.7,0)node{$0$};
\draw(7,-.7)node{$p$};\draw(0,-.7)node{$0$};
\foreach \x in {0,...,7}\foreach \y in {\x,...,7} { \fill (\x,7-\y) circle(.2); }
\foreach \y in {1,...,6}\foreach \x in {\y,...,6} { \draw[thick](6-\x,\y)--(8-\x,\y-1); }
\draw(7.8,7.5)node{$C_{j_x\,j_t}$};
\end{tikzpicture}
}
\hfill
\subfloat[Case $V \neq 0$]{
\begin{tikzpicture}[scale=.5]
\fill[lightgray](-.4,-.4)--(1.4,-.4)--(1.4,7.4)--(-.4,7.4);
\draw[step=1cm,gray,very thin] (0,0) grid (7,7);
\draw[very thick,->](0,0)--(8,0);\draw[very thick,->](0,0)--(0,8);
\draw(8.7,0)node{$j_x$};\draw(-1,7.8)node{$j_t$};
\draw(-.7,7)node{$p$};\draw(-.7,0)node{$0$};
\draw(7,-.7)node{$p$};\draw(0,-.7)node{$0$};
\foreach \x in {0,...,7}\foreach \y in {\x,...,7} { \fill (\x,7-\y)circle(.2); }
\foreach \y in {1,...,6}\foreach \x in {\y,...,6} 
    { \draw[thick](6-\x,\y)--(6.5-\x,\y-.5)--(8-\x,\y-1);
    \draw[thick](6-\x,\y-1)--(6.5-\x,\y-.5);
    \draw[fill=white,thick](6.5-\x,\y-.5) circle(.15);}
\draw(7.8,7.5)node{$C_{j_x\,j_t}$};
\fill (7.5,4) circle(.2); \draw(9.1,4)node{\quad coefficient};
\draw[fill=white,thick](7.5,3) circle(.15);\draw(8.8,3)node{\quad relation};
\end{tikzpicture}
}
\caption[]{Illustration of the relations defining the set $\calD$ in the 
(1+1)-dimensional case. 
The black dots in the $(j_x,j_t)$ plane represent the coefficients $C_{j_x\,j_t}$ of $\bC\in\calD\subset\IC^{r_p}$, $r_p=\frac{(p+1)(p+2)}2$.
In the case $V\ne0$ (right panel), each shape 
\begin{tikzpicture}[scale=.35]
\fill (0,0)circle(.2); \fill (2,0)circle(.2); \fill (0,1)circle(.2); 
\draw (0,0)--(.5,.5)--(2,0);
\draw (0,1)--(.5,.5);
\draw[fill=white,thick](.5,.5) circle(.15);
\end{tikzpicture}
connects three black dots located at the points $(j_x,j_t+1)$, $(j_x,j_t)$ and $(j_x+2,j_t)$: this shape represents one of the equations \eqref{EQN::TAYLOR-COEFFICIENTS-RELATION} which, given $C_{j_x(j_t+1)}$ and $C_{j_x\,j_t}$, allows to compute $C_{(j_x+2)j_t}$.
If the $2p+1$ values with $j_x\in\{0,1\}$ (corresponding to the nodes in the 
shaded region) are given, then these relations uniquely determine all the other 
coefficients, which can be computed sequentially by proceeding left to right in 
the diagram. 
In the figure $p=7$, the number of nodes is $r_p=36$, the number of nodes in the 
shaded region is $n_{1,p}=15$, the number of relations is $r_p-n_{1,p}=21$.
The case $V=0$ (left panel) is slightly simpler: the coefficient $C_{(j_x+2)j_t}$ is determined by $C_{j_x(j_t+1)}$ only and each relation in \eqref{EQN::TAYLOR-COEFFICIENTS-RELATION} is depicted as a 
segment~\begin{tikzpicture}[scale=.3]
\fill (2,0)circle(.2); \fill (0,1)circle(.2); 
\draw (0,1)--(2,0);
\end{tikzpicture}.
\label{FIG::INDEPENDENT-RELATIONS}}
\end{figure}

The rest of the proof consists in proving that $\RANK{\bM} \geq 2p + 1${, i.e.\ that $\bM$ is full-rank}.
To do so, we first
recall that $\phi_{\ell}(x, t) = \exp[i(k_\ell x - \left(k_{\ell}^2 + 
V\right)t)]$ from the basis definition \eqref{EQN::BASIS-FUNCTIONS}, and
observe that the entries of matrix $\bM$ are given by
$$
M_{\bj,\ell} = \frac{1}{j_x!j_t!} \DerA{\bj}{\phi_{\ell}}(z,s) = 
q_{j_x,j_t}(k_\ell) \phi_{\ell}(z,s),\quad \bj = (j_x, j_t), \quad \ell = 1, 
\ldots, 2p + 1,
$$
where $q_{j_x,j_t}(k) = \frac{1}{{j_x!j_t!}} (ik)^{j_x} \left(-i\left(k^2 + V\right)\right)^{j_t}$
is a complex-valued polynomial of degree exactly $j_x+2j_t$.
We define a square submatrix $\bM_{\msquare}$ of the matrix $\bM$ by taking the rows corresponding to $j_x \in\{0, 1\}$. This, in 
turn, can be decomposed as 
$\bM_{\msquare} = \bV \bD_{z,s}$, 
where $\bD_{z,s} = \mbox{diag}(\phi_1(z, s), \ldots, \phi_{2p + 1}(z, s))$ 
and the Vandermonde-like matrix $\bV \in \IC^{(2p + 
1)\times (2p + 1)}$ is given by
\begin{align}
\label{EQN::VANDERMONDE-MATRIX}
\bV & = \left(
\begin{tabular}{cccc}
$q_{0,0}(k_1)$ & $q_{0,0}(k_2)$ & \ldots & 
$q_{0,0}(k_{2p+1})$ \\
$q_{1,0}(k_1)$ & $q_{1,0}(k_2)$ & \ldots & 
$q_{1,0}(k_{2p+1})$ \\
\vdots & \vdots &  & \vdots \\
$q_{0,j_t}(k_1)$ & $q_{0,j_t}(k_2)$ & \ldots & 
$q_{0,j_t}(k_{2p+1})$ \\
$q_{1,j_t}(k_1)$ & $q_{1,j_t}(k_2)$ & \ldots & 
$q_{1,j_t}(k_{2p+1})$ \\
\vdots & \vdots &  & \vdots \\
$q_{0,p-1}(k_1)$ & $q_{0,p-1}(k_2)$ & \ldots & 
$q_{0,p1}(k_{2p+1})$ \\
$q_{1,p-1}(k_1)$ & $q_{1,p-1}(k_2)$ & \ldots & 
$q_{1,p-1}(k_{2p+1})$ \\
$q_{0,p}(k_1)$ & $q_{0,p}(k_2)$ & \ldots & 
$q_{0,p}(k_{2p+1})$ \\
\end{tabular}
\right).
\end{align}   
We observe that $V_{m, \ell} = p_{m-1}(k_\ell),\ m = 1, \ldots, 2p + 1$, for some polynomials $p_m(\cdot)$ of degree $m$.
Therefore there exists a lower triangular matrix $\bP$ such that 
$\bP\bV=\widetilde\bV$, where $\widetilde V_{m,\ell}=k_\ell^{m-1}$ for 
$m,\ell=1,\ldots,2p+1$.
(The entries of the $m$th row of the inverse of $\bP$ are the coefficients of the monomial expansion of the polynomial $p_{m-1}$.)
This means that $\widetilde\bV$ is a Vandermonde matrix, so it is invertible \cite[\S22.1]{Higham_2002} (recall that all $k_{\ell}$ are different from one another).
We conclude that the matrix $\bV$ is invertible, independent of $(z,s)$ and 
$\bM$ is full rank (namely, $\RANK\bM=2p+1$).

Denoting by $\bb_{\msquare}$ the subvector of $\bb$ corresponding to the 
indices $j_x \in\{0,1\}$, the unique solution of the linear system $\bM \ba(z,s) 
= \bb$ is $\ba(z,s) = \bM_{\msquare}^{-1} \bb_{\msquare}$ and satisfies 
condition~\eqref{EQN::Condition1}. Moreover, the following bound 
holds
\begin{equation}
\label{EQN::BOUND-A-1NORM}
\abs{\ba(z,s)}_1 \leq \MatrixNorm{\bV^{-1}}{1} 
\MatrixNorm{\bD_{z,s}^{-1}}{1}\abs{\bb_{\msquare}}_{1},
\end{equation}
where $\MatrixNorm{\bD^{-1}_{z,s}}{1} = 1$\ for all 
$(z,s)\in B$.
We now recall that $(z,s)$ was chosen arbitrarily in $B\setminus\Upsilon$.
Setting $\ba(z,s)=\bzero$ and $\bb_{\msquare}(z,s)=\bzero$ for $(z,s)\in 
\Upsilon$ (which has zero measure), recalling from \eqref{EQN::ATP-PSI} that 
$b_\bj=\frac1{\bj!}D^\bj\psi$, and integrating \cref{EQN::BOUND-A-1NORM} over 
$B$, we obtain  
\begin{align*}
\N{|\ba|_1}_{L^1(B)}
&\le \MatrixNorm{\bV^{-1}}1 \N{\abs{\bb_{\msquare}}_1}_{L^1(B)}
\le \MatrixNorm{\bV^{-1}}1 \sqrt{2p+1}\N{\abs{\bb_{\msquare}}_2}_{L^1(B)}\\
&\le \MatrixNorm{\bV^{-1}}1 \sqrt{2p+1} \sqrt{|B|}\N{\psi}_{H^{p+1}(B)}.
\end{align*}
This implies the assertion \cref{EQN::Condition2} with 
$C_\star=\MatrixNorm{\bV^{-1}}{1} \sqrt{2p+1}$, since $|B|\le|K|$.
\end{proof}

{Since the constant $C_\star$ depends on the norm of the inverse of the Vandermonde matrix $\bV$, which in turn depends on the parameters $\left\{k_\ell\right\}_{\ell = 1}^{2p + 1}$, the choice of the parameters defining the Trefftz basis functions has an important influence on the accuracy of the method; this is illustrated in the numerical experiments in \Cref{s:NumExp1D}.
}

\subsection{Best approximation in the (2+1)-dimensional case}
\label{SECT::APPROX-2D}
Following the same strategy of
the previous section, in \Cref{PROP::POINTWISE-TAYLOR-2D} we show that, for a 
sensible choice of the parameters $\left\{k_\ell\right\}$ and the directions 
$\left\{\bd_{\ell}\right\}$ in \eqref{EQN::BASIS-FUNCTIONS}, 
\Cref{ConditionApprox} is true for $\bVp(\Th)$ in the $(2 + 1)$-dimensional case. 

The basis functions $\phi_{m,\lambda}$ are plane waves in space indexed by two 
parameters: $m$ identifying the wavenumber $k_m$, and $\lambda$ identifying the 
propagation direction $\theta_{m,\lambda}$.
For every wavenumber we take a different number of directions following a strategy similar to that used for the construction of plane-wave Trefftz bases for the 3D Helmholtz equation in \cite[Lemma~4.2]{Moiola_Hiptmair_Perugia_2011}.
The time-dependence of each basis element is harmonic with frequency $k_m^2+V|_K$.
A sample basis   for $p = 2$ is shown in \cref{fig:2Dbasis}.

\begin{proposition}
\label{PROP::POINTWISE-TAYLOR-2D}
Let $d=2$ and $n_{2,p}=(p+1)^2$.
Let the parameters $k_m$ and $\theta_{m,\lambda}$ satisfy the following 
conditions:
\begin{align*}
k_m\in \IR \; &\tfor m=0,\ldots,p,\; \text{ with }\; k_{m_1}^2\ne k_{m_2}^2 
\tfor m_1\ne m_2 
\;\tand \; k_m \neq 0,
\\
\theta_{m, \lambda}\in [0,2\pi) \; &\tfor m=0,\ldots,p, \;
\lambda =1,\ldots,2m+1, 
\; \text{ with }\; \theta_{m, \lambda_1}\ne \theta_{m, \lambda_2}\tfor 
\lambda_1\ne \lambda_2.\!
\end{align*}
Define the directions 
$\bd_{m,\lambda}=(\cos\theta_{m,\lambda},\sin\theta_{m,\lambda})$ and the 
basis functions
$$
\phi_{m,\lambda} (\bx,t)=
\exp\left[i\left(k_m \bd_{m,\lambda}^\top \bx - (k_{m}^2 + V|_{K})t\right)\right]
\ 
\tfor m=0,\ldots,p,\; 
\lambda=1,\ldots,2m+1.
$$
Then \Cref{ConditionApprox} holds true.
\end{proposition}

\begin{figure}[htb]
\includegraphics[width=\textwidth,clip,trim=115 20 105 10]{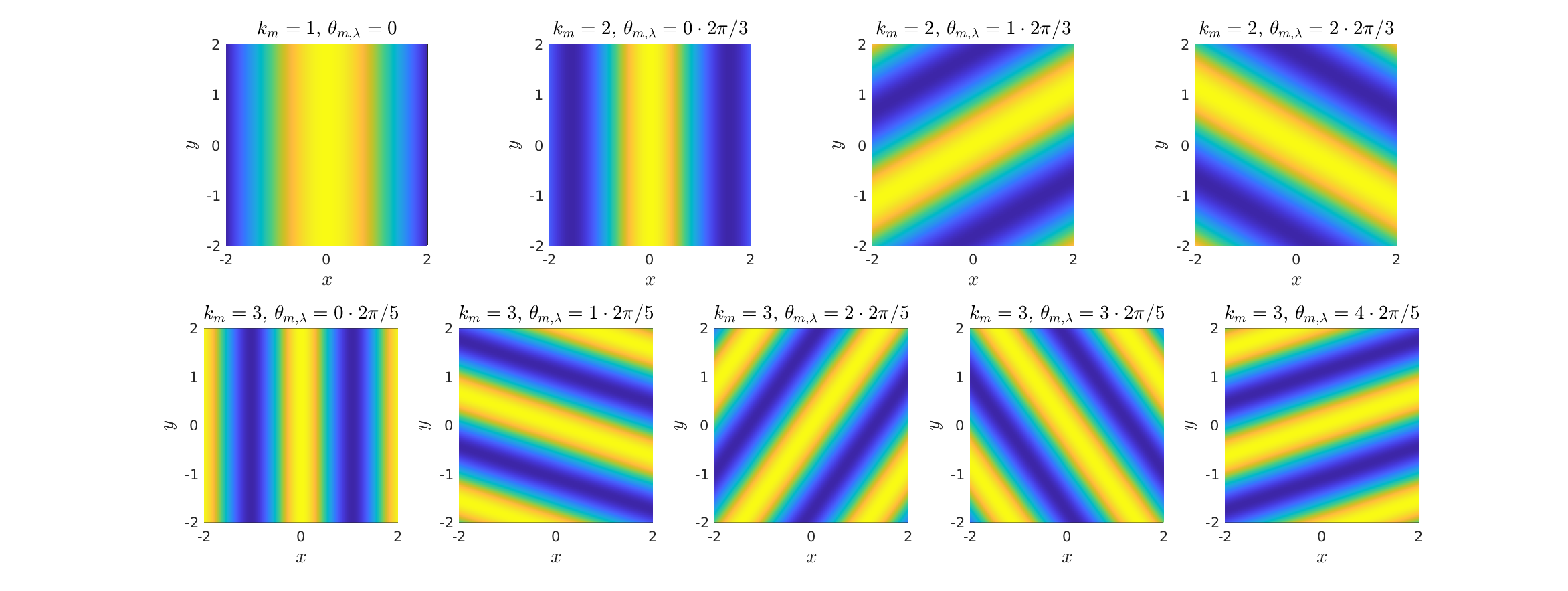}
\caption{The real parts of the Trefftz basis functions defined in 
\eqref{PROP::POINTWISE-TAYLOR-2D} for $V=0$ and $p=2$ at time $t=0$ on the space 
domain $(-2,2)^2$.
Here $k_m=m+1$ for $m=0,1,2$ and $\theta_{m,\lambda}=2\pi\frac{\lambda-1}{2m+1}$ for $\lambda=1,\ldots,2m+1$.
Note that, differently from Figure~\ref{fig:1Dbasis}, here we only plot the space dependence of the $\phi_\ell$.
}
\label{fig:2Dbasis}
\end{figure}

\begin{proof} 
Let $\psi\in \bT(K)\cap H^{p+1}(K)$ and $(\bz,s)\subset B\setminus\Upsilon$, 
with the set $\Upsilon$ as in \cref{rem:PointDerivatives}.
As in \cref{PROP::POINTWISE-TAYLOR-1D}
we can arrange the coefficients of the 
Taylor polynomials of $\psi$ and each basis function $\phi_{\ell}$ in a vector 
$\bb \in \IC^{r_p}$ and a matrix $\bM \in \IC^{r_p \times n_{2,p}}$ with $r_p = 
(p+1)(p+2)(p+3)/6$. 
In order to prove \eqref{EQN::Condition2}, we look for a vector $\ba(\bz, s) 
\in \IC^{n_{2,p}}$ such that $\bM \ba(\bz,s) = \bb$.

Similarly to the $(1 + 1)$ dimensional case, $\bb \in \calD$ and $\IMG{\bM} 
\subset \calD$ for $\calD$
the set of all the vectors $\bC \in \IC^{r_p}$ satisfying the following 
relations
\begin{align}\nonumber
i(j_t + 1) C_{j_x j_y (j_t + 1)} \ & +\  (j_x + 1)(j_x + 2) C_{(j_x + 2)j_y 
j_t} \\
& +  (j_y + 1)(j_y + 2) C_{j_x (j_y + 2) j_t}= VC_{j_x j_y j_t}
\label{EQN::TAYLOR-COEFFICIENTS-RELATIONS-2D}\\
& \hspace{3cm}\tfor |\bj| \leq p - 2, \ \bj = (j_x, j_y, j_t).
\nonumber
\end{align}
\Cref{FIG::Relations2+1D} depicts the equations 
\eqref{EQN::TAYLOR-COEFFICIENTS-RELATIONS-2D} as relations between the 
coefficients of the vector $\bC$, which are represented as points in the 
$(j_x,j_y,j_t)$ 
space.
In particular, it shows that the $(p+1)^2$ entries of any $\bC\in\calD$ with $j_x\in\{0,1\}$ determine all the other entries of $\bC$, thus 
$\RANK\bM\le\DIM\calD\le(p+1)^2$.

\begin{figure}[htb]
\centering
\begin{tikzpicture}
[rotate around x=-90,rotate around y=0,rotate around z=-125,grid/.style={very thin,gray},scale=1]

\def\maxX{5}
\foreach \x in {0,1,...,\maxX}
\foreach \y in {0,1,...,\maxX}
\foreach \z in {0,1,...,\maxX}
    {
    \draw[grid,lightgray] (\x,0,0) -- (\x,\maxX,0);
    \draw[grid,lightgray] (0,\y,0) -- (\maxX,\y,0);
    \draw[grid,lightgray] (0,\y,0) -- (0,\y,\maxX);
    \draw[grid,lightgray] (0,0,\z) -- (0,\maxX,\z);
    \draw[grid,lightgray] (\x,0,0) -- (\x,0,\maxX);
    \draw[grid,lightgray] (0,0,\z) -- (\maxX,0,\z);
    }

\foreach \x in {0,...,\maxX} 
    {
    \pgfmathsetmacro\xx{\maxX-\x}  
    \foreach \y in {0,...,\xx}
        {
        \pgfmathsetmacro\xxy{\maxX-\x-\y}  
        \foreach \z in {0,...,\xxy}
            {  
            \ifnum \x>0 \relax  \ifnum \y>0 \relax  \ifnum \z>0 \relax
                \draw [lightgray,dashed] (\x,\y,0)--(\x,\y,\z);
                \draw [lightgray,dashed] (\x,0,\z)--(\x,\y,\z);
                \draw [lightgray,dashed] (0,\y,\z)--(\x,\y,\z);
            \fi \fi \fi
            \ifnum \x<2\relax
                \draw[blue] plot [mark=*, mark size=2] coordinates{(\x,\y,\z)}; 
            \else
                \draw plot [mark=*, mark size=2] coordinates{(\x,\y,\z)}; 
            \fi
            }
        }
    }
    
\pgfmathsetmacro\maxXX{\maxX-2} 
\foreach \x in {0,...,\maxXX}
    {
    \pgfmathsetmacro\xx{\maxXX-\x}
    \foreach \y in {0,...,\xx}
        {
        \pgfmathsetmacro\xxy{\maxXX-\x-\y}  
        \foreach \z in {0,...,\xxy}
            {        
            \draw [very thick](\x,\y,\z)--(\x+.5,\y+.5,\z+.5)--(\x,\y,\z+1);
            \draw [very thick] (\x+2,\y,\z)--(\x+.5,\y+.5,\z+.5)--(\x,\y+2,\z);
            \draw [fill=white, thick] plot [mark=*, mark size=3] coordinates{(\x+.5,\y+.5,\z+.5)};      
            }
        }
    }
\draw(\maxX+.4,0,0)node{$j_x$};
\draw(0,\maxX+.4,0)node{$j_y$};
\draw(0,.5,\maxX+0.2)node{$j_t$};

\draw [red,very thick](0,1,2)--(.5,1.5,2.5)--(0,1,2+1);
\draw [red,very thick](2,1,2)--(0+.5,1.5,2.5)--(0,1+2,2);
\draw [red,fill=white,thick] plot [mark=*, mark size=3] coordinates{(.5,1.5,2.5)}; 
\end{tikzpicture}
\caption{A representation of the relations defining the set $\calD$ in the 
(2+1)-dimensional case.
The colored dots in position $\bj=(j_x,j_y,j_t)$, $|\bj|\le p$, correspond to the entries $C_{j_x\,j_y\,j_t}$ of the vector $\bC\in\calD\subset\IC^{r_p}$ (here $p=5$ and $r_p=56$).
Each white circle is connected by the segments to four nodes and represents one of the equations in \eqref{EQN::TAYLOR-COEFFICIENTS-RELATIONS-2D}: given $C_{j_x\,j_y\,j_t}$, $C_{j_x\,j_y(j_t+1)}$ and $C_{j_x(j_y+2)j_t}$, it allows to compute $C_{(j_x+2)j_y\,j_t}$ (the leftmost of the four nodes connected to a given white circle).
The red dot exemplifies one of these relations, for $\bj=(0,1,2)$.
Given the $(p+1)^2$ coefficients with $j_x\in\{0,1\}$ (the blue dots), all other 
coefficients are uniquely determined and $\DIM\calD\le (p+1)^2$.
\label{FIG::Relations2+1D}}
\end{figure}
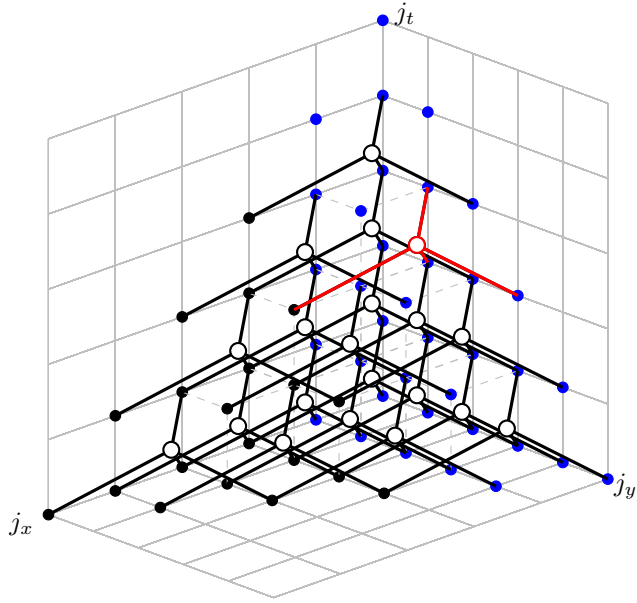

Now it just remains to prove that $\RANK{\bM} \geq (p + 1)^2$.

We fix an arbitrary ordering $\ell=1,\ldots,(p+1)^2$ of the ``triangular'' index 
set 
$\{(m,\lambda): m=0,\ldots,p,\; \lambda=1,\ldots,2m+1\}$, so that we 
can write $\phi_\ell$, $k_\ell$ and $\theta_\ell$ for the basis 
functions $\phi_{m,\lambda}$ and the parameters $k_{m,\lambda}$ and 
$\theta_{m,\lambda}$.

First, from the basis function definition, the matrix $\bM$ can be decomposed as 
$\bM 
= \bG \bD_{\bz, s}$ for 
the diagonal matrix $\bD_{\bz, s} = \mbox{diag}(\phi_{1}(\bz, s), \ldots, 
\phi_{(p+1)^2}(\bz, s))$ and a matrix $\bG \in \IC^{r_p, n_{2,p}}$ with entries 
independent of $(\bz, s)$ given by:
\begin{equation}
\bG_{\bj, \ell} = \frac{1}{j_x!j_y! j_t!} 
\big(ik_\ell \cos\theta_{\ell}\big)^{j_x}
\big(ik_\ell \sin\theta_{\ell}\big)^{j_y}
\big(-i(k_\ell^2 + V)\big)^{j_t},
\quad \bj = (j_x, j_y, j_t).
\end{equation}
We define a square matrix $\bS \in \IC^{n_{2,p} \times n_{2, p}}$ with the 
block structure $\bS =\begin{pmatrix} \bS^+ \\\bS^- \end{pmatrix}$, 
where $\bS^+ \in \IC^{\frac{(p+1)(p+2)}2 \times n_{2, p}}$ and $\bS^- \in 
\IC^{\frac{p(p+1)}2 \times n_{2, p}}$ are defined as
\begin{align*}
S^+_{(j_\bx,j_t), \ell} & = \left(k_\ell^2 + V\right)^{j_t} k_\ell^{j_\bx} 
\ee^{i \theta_{\ell} j_\bx}  & \; \tfor
\begin{tabular}{ll}
$j_\bx = 0 , \ldots, p,$\\
$j_t = 0, \ldots, p - j_{\bx}$,
\end{tabular}\ell = 1, \ldots, (p+1)^2,\\
S^-_{(j_\bx,j_t), \ell} & = \left(k_\ell^2 + V\right)^{j_t} k_{\ell}^{j_\bx}
\ee^{-i \theta_{\ell} j_\bx} & \; \tfor
\begin{tabular}{ll}
$j_\bx = 1, \ldots, p,$\\
$j_t = 0, \ldots, p - j_{\bx}$,
\end{tabular}\ell = 1, \ldots, (p+1)^2.
\end{align*}
The binomial theorem {allows us} to relate the blocks $\bS^\pm$ to $\bG$:
\begin{align*}
S_{(j_\bx,j_t), \ell}^+ & = \left(k_\ell^2 + V\right)^{j_t} (k_\ell 
\cos\theta_\ell + i k_\ell 
\sin\theta_\ell)^{j_\bx} \\
& = \sum_{j_x + j_y = j_{\bx}} 
\frac{j_\bx !}{j_x!j_y!} (k_\ell \cos \theta_\ell)^{j_x}
(ik_\ell\sin\theta_\ell)^{j_y} (k_\ell^2 + V)^{j_t}\\
& = \sum_{j_x + j_y = j_{\bx}} \frac{j_t! \:{j_{\bx}!} }{
i^{j_x}(-i)^{j_t}} \bG_{\bj, \ell}, 
&&\hspace{-5mm}\begin{tabular}{ll}
$j_\bx = 0 , \ldots, p,$\\
$j_t = 0, \ldots, p - j_{\bx}$,\\
$\ell=1,\ldots, (p+1)^2$,
\end{tabular}
\\
S_{(j_\bx,j_t), \ell}^- & = \left(k_\ell^2 + V\right)^{j_t} 
(k_\ell\cos\theta_\ell - i k_\ell \sin\theta_\ell)^{j_\bx} \\
& = \sum_{j_x + j_y = j_{\bx}} \frac{j_\bx !}{j_x!j_y!}(k_\ell \cos\theta_\ell)^{j_x}
(-ik_\ell\sin\theta_\ell)^{j_y} \left(k_\ell^2 + V \right)^{j_t} \\
& = \sum_{j_x + j_y = j_{\bx}} \frac{j_t! \:{j_{\bx}!}}{i^{j_x} (-1)^{
j_y}(-i)^{j_t}} \bG_{\bj, \ell},
&&\hspace{-5mm}\begin{tabular}{ll}
$j_\bx = 1 , \ldots, p,$\\
$j_t = 0, \ldots, p - j_{\bx}$,\\
$\ell=1,\ldots, (p+1)^2$.
\end{tabular}
\end{align*}

This means that there exists a matrix $\bP \in \IC^{n_{2,p} \times r_p}$ such 
that 
$\bS = \bP \bG$.

Next, we prove that $\bS$ is not singular. 
If $\bS^\top \bc =\uu{0}$ for some vector $\bc \in \IC^{n_{2,p}}$, then for 
each pair $\left( k_m,\ \theta_{m, \lambda}\right)$ in the definition of 
the basis functions we have 
\begin{align}
\label{EQN::TRIGONOMETRIC-POLYNOMIAL}
0 = \sum_{n = -p}^p \bigg(\sum_{j_t = 0}^{p - \abs{n}} c_{n, j_t} 
\left(k_{m}^{2} + V\right)^{{j_t}}k_m^{\abs{n}}\bigg) 
\ee^{in\theta_{m,\lambda}}
= \sum_{n = -p}^p \zeta_{n}(k_m)\ \ee^{i n\theta_{m,\lambda}},
\\
m=0,\ldots,p,\quad\lambda=1,\ldots,2m+1,
\nonumber
\end{align}
If we fix $m = p$ in \Cref{EQN::TRIGONOMETRIC-POLYNOMIAL}, the $\zeta_{n}(k_p)$ 
are the 
coefficients of a trigonometric polynomial of degree $p$ with $2 p + 1$ 
different zeros $\theta_{p, \lambda}$, which implies that $\zeta_{n}(k_p) = 
0,$ 
for $n = -p, \ldots, p$.
In particular $ 0 = \zeta_{\pm p}(k_p) = c_{\pm p, 0} k_p^p 
\Rightarrow c_{\pm p, 0} = 0,$ since $k_p \neq 0$ by hypothesis. 

We now proceed by (backward) induction.
Assume that for some $\eta \in \left\{1, \ldots, p\right\}$ we have:
\begin{subequations}
\begin{align}
\label{EQN::COEFFICIENTS-TRIG-POLY-1}
c_{\pm n, j_t} & = 0,\ \mbox{ for } n = p - \eta + 1, \ldots, p \
\tand \ j_t=0,\ldots,p-|n|,\\
\label{EQN::COEFFICIENTS-TRIG-POLY-2}
\zeta_n(k_m) & = 0,\ \mbox{ for } n = -p, \ldots, p\ \mbox{ and }\ 
m = p - \eta + 1, \ldots, p.
\end{align}
\end{subequations}
Then \eqref{EQN::TRIGONOMETRIC-POLYNOMIAL} gives
\begin{equation*}
\sum_{n = - (p - \eta)}^{p - \eta} \zeta_n(k_m) 
\ee^{i n\theta_{m,\lambda}} = 0, 
\end{equation*}
and for $m = (p - \eta)$, the $\zeta_n(k_{p-\eta})$ are the coefficients of a 
trigonometric polynomial of degree $p - \eta$ with $2(p - \eta) + 1$ different 
zeros $\theta_{p - \eta, \lambda}$.
Therefore assumption \cref{EQN::COEFFICIENTS-TRIG-POLY-2} also holds for $m = p 
- \eta$.
For $n = \pm (p - \eta)$ in \cref{EQN::COEFFICIENTS-TRIG-POLY-2} we have
\begin{equation*}
	k_{m}^{p - \eta} \sum_{j_t = 0}^{\eta} c_{\pm (p - \eta), j_t} 
\left(k_{m}^2 + V\right)^{j_t} = 0;
\end{equation*}
therefore $c_{\pm (p - \eta), j_t}$ are the coefficients of a complex-valued 
polynomial of degree $\eta$ with $\eta + 1$ different zeros $\left(k_{m}^2 + 
V\right)$,\ $m = p - \eta, \ldots, p$; which implies that assumption 
\cref{EQN::COEFFICIENTS-TRIG-POLY-1} holds for $n = p - \eta$. 
Recursively, this leads to $\bc = \uu{0}$; therefore the matrix $\bS$ is invertible
and, since $\bS=\bP\bG$ and $\bM=\bG\bD_{\bz,s}$ for a nonsingular 
$\bD_{\bz,s}$, the matrix $\bM$ has rank at least $n_{2,p}$ so it is full 
rank.

The solution $\ba(\bz, s)$ of the rectangular linear system $\bM \ba(\bz, s) = 
\bb$ is 
$\ba(\bz, s) = \bD_{\bz, s}^{-1} \bS^{-1} \bP \bb$, satisfies 
\eqref{EQN::Condition1}, and the following 
bound is obtained
\begin{equation*}
\abs{\ba(z,s)}_1 \leq 
\MatrixNorm{\bS^{-1}}{1}\MatrixNorm{\bD_{\bz, s}^{-1}}{1} \abs{\bP \bb}_{1},
\end{equation*}
where $\MatrixNorm{\bD^{-1}_{z,s}}{1} = 1$\ for all $(\bz,s)\in B$.
Writing $b_\bj=\frac1{\bj!}D^\bj\psi(\bz,s)$ and integrating over $B$ as in the 
proof of \cref{PROP::POINTWISE-TAYLOR-1D}, we obtain \cref{EQN::Condition2} with 
$C_\star=\MatrixNorm{\bS^{-1}}1\MatrixNorm{\bP}1 \sqrt{r_p}$.
\end{proof}

{Analogously to the $(1 + 1)$-dimensional case, the constant $C_\star$ depends on the norm of the inverse of the matrix $\bS$, which depends on the choice of both the parameters $\{k_m\}$ and the directions $\left\{\bd_{m,\lambda}\right\}$.
This indicates that finding an appropriate choice of these parameters is crucial in order to get accurate and stable approximations.}

\begin{remark}\label{Rem:3D}
{In order to extend \Cref{ConditionApprox}, and thus the Trefftz approximation theory, to $(d+1)$-dimensional problems with $d\ge3$, one has to provide conditions on the parameters defining the local basis functions.
Not only one has to determine the minimal number of directions $\{\bd_{m,\lambda}\}\subset\mathcal{S}^d_1$ associated to each parameter $k_m$, but also the mutual relations that the directions have to satisfy (e.g.\ not too many of them can belong to the same $(d-1)$-dimensional hyperplane).
Work in this direction is currently in progress.
}
\end{remark}


\subsection{Error bounds in DG norm} \label{SECT::DG-ERROR-ESTIMATE}

The next proposition provides a bound on the $\Tnorm{\cdot}{\DGp}$ norm in 
terms 
of volume Sobolev seminorms and norms and is a direct consequence of the trace 
inequalities \cref{EQN::TRACE-INEQUALITY-TIME}.

\begin{proposition}\label{PROP::DGp-BOUND}
For all $\varphi\in \prod_{K\in\calT_h(Q)} H^1(I_n;L^2(\Kx))\cap 
L^2(I_n;H^2(\Kx))$,
the following bound holds
\begin{align*}
\Tnorm{\varphi}{\DGp}^2 \leq  3\,C_{\tr}\sum_{K \in \Th(Q)} \Big[&
h_n^{-1} \Norm{\varphi}{L^2(K)}^2 + h_n 
\Norm{\partial_t \varphi}{L^2(K)}^2\\
& +\mathrm{a}_K^2 h_{\Kx}^{-1}\Norm{\varphi}{L^2(K)}^2 + \mathrm{a}^2_K h_{\Kx} 
\Norm{\nabla \varphi}{L^2(K)^d}^2 \\
& + \mathrm{b}_K^2 h_{\Kx}^{-1}\Norm{ \nabla \varphi}{L^2(K)^d}^2 + 
\mathrm{b}_K^2 h_{\Kx} \Norm{D^2 \varphi}{L^2(K)^{d\times d}}^2 \Big],
\end{align*}
where $D^2 \left(\varphi \right)$ is the space--time Hessian of $\varphi$ and 
\[
\begin{tabular}{ll}
$\alpha_{\inf}^K := \underset{\partial K \cap \left(\Ftime \cup 
\FD\right)}{\essinf} \alpha,$
& $\beta_{\inf}^K := \underset{\partial K \cap \Ftime}{\essinf}\beta$, \\
$\alpha_{\sup}^K := \underset{\partial K \cap \left(\Ftime \cup 
\FD\right)}{\esssup} \alpha,$
& $\beta_{\sup}^K := \underset{\partial K \cap \Ftime}{\esssup}\beta$, \\
$\mathrm{a}_K := \max\left\{\left(\alpha_{\sup}^K \right)^{1/2},\ 
\left(\beta_{\inf}^K \right)^{-1/2}\right\},$
& $\mathrm{b}_K := \max\left\{\left(\alpha_{\inf}^K \right)^{-1/2},\ 
\left(\beta_{\sup}^K \right)^{1/2}\right\}.$
\end{tabular}
\]
\end{proposition}

The factor 3 appearing in the bound of \Cref{PROP::DGp-BOUND} is due to the 
integral terms with arguments $\jmp{w}_t^2+(w^-)^2$ on $\Fspa$, 
$\alpha\abs{\jmp{w}_\bN}{}^2 + \beta^{-1}\mvl{w}^2$ and $\beta\jmp{\nabla 
w}_\bN^2+\alpha^{-1} \abs{\mvl{\nabla w}}^2$ on $\Ftime$ in the definition 
\eqref{EQN::DG-NORMS} of the $\Tnorm{\cdot}\DG$ norm.

\Cref{THM::ERROR-ESTIMATE} provides the error estimate for the Trefftz-DG 
approximation of \eqref{EQN::SCHRODINGER-EQUATION} in the 
$\Tnorm{\cdot}\DG$ norm assuming that \Cref{ConditionApprox} holds true.
It is consequence of 
\Cref{PROP::DGp-BOUND,PROP::H-ESTIMATE}, 
\Cref{THEOREM::WELL-POSEDNESS} and the 
``local quasi-uniformity in space'' assumption on the mesh.

\begin{theorem}
\label{THM::ERROR-ESTIMATE}
Let $p\in\IN$. Let $\psi \in \bT(\Th) 
\cap H^{p + 1}(\Th)$ be the exact solution of 
\eqref{EQN::SCHRODINGER-EQUATION} and $\uhp\in \bVp(\Th)$ be the Trefftz-DG 
approximation solving \cref{EQN::VARIATIONAL-TREFFTZ-DG} with $\bVp(\Th)$ 
satisfying \Cref{ConditionApprox} for all $K \in \Th(Q)$.
Set the stabilization parameters as
\begin{align*}
&\alpha\big|_{F} = \frac1{h_{F_\bx}}
\quad \forall F \subset \Ftime\cup\FD, 
\qquad \beta\big|_{F} = h_{F_\bx} \quad \forall F \subset \Ftime,\; 
\text{where}
\\
&\tfor F\subset \deK\cap \FD:\quad h_{F_\bx}=h_{\Kx}
\\
&\tfor F=\Kx^1\times(t_{n-1},t_n)\;\cap\;\Kx^2\times(t_{n-1},t_n)\subset \Ftime:\\
&\hspace{10mm}\text{$h_{F_\bx}$ is any constant satisfying:}
\quad \min\{h_{\Kx^1},h_{\Kx^2}\}\le h_{F_\bx}\le 
\max\{h_{\Kx^1},h_{\Kx^2}\}.
\end{align*}
Then there exists a constant $C$ independent on the mesh size such that
$$
\Tnorm{\psi - \uhp}{\DG} \leq C \sum_{K=\Kx\times(t_{n-1},t_n)\in\calT_h(Q)} 
\max\{h_{\Kx},h_n\}^p \N{\psi}_{H^{p+1}(K)}.
$$
\end{theorem}

\begin{remark}
\label{REMARK::ENERGY-ESTIMATE}
The last formula in \cref{Rem:Energy} shows that, if $\gD=0$ and 
the assumptions 
of \cref{THM::ERROR-ESTIMATE} are satisfied, the energy dissipated by the 
Trefftz DG method (i.e.\ $\calE(0;\psi_0)-\calE( T,\psi\hp)$) 
converges 
 to zero proportionally to the square of the error in the DG norm, i.e.\ as
$\max_{K\in\Th(Q)}\max\{h_{K_\bx},h_n\}^{2p}$.
\end{remark}

\begin{remark}\label{Rem:LessDOFs}
The previous sections show that for $d=1,2$ the local space $\IT^p(K)$ has dimension $\calO_{p\to\infty}(p^d)$ and approximates Schr\"odinger solutions with the same rates of the space $\IP^p(K)$ of the degree-$p$ polynomials on $K$, which has larger dimension $\dim\IP^p(K)=\calO_{p\to\infty}(p^{d+1})$.
More precisely, $\IT^p(K)$ has the dimension of the space of the harmonic polynomials of degree $p$ on $K$.
We expect the same to hold for $d>2$.
This is a major advantage of the Trefftz approach: it achieves the same convergence rates of standard methods with considerably fewer degrees of freedom.
The same situation is well-known for other PDEs, see e.g.\
\cite{Moiola_Hiptmair_Perugia_2011},
\cite[Fig.~4]{Perugia_Schoeberl_Stocker_Wintersteiger_2020},
\cite[Rem.~4.11]{ImbertGerard_Moiola_Stocker}.
\end{remark}


\section{Numerical experiments\label{SECT::NUMERICAL-EXPERIMENTS}}

We present some numerical experiments validating the error 
estimates in the mesh-dependent norm $\Tnorm{\cdot}\DG$ derived 
in \Cref{SECT::DG-ERROR-ESTIMATE}. We also numerically assess the error 
behavior with respect to $h := \max \left\{h_\bx, h_t\right\}$ 
in the final-time, mesh-independent $\Norm{\cdot}{L_2(\FT)}$ norm and evaluate 
the energy dissipation of the proposed method.
All experiments have been implemented in Matlab.

A direct implementation of the variational problem 
\eqref{EQN::VARIATIONAL-TREFFTZ-DG} leads to 
a large global linear system involving all the degrees of freedom of the 
expansion coefficients of $\uhp$ in the basis of $\bVp(\Th)$, over the full 
space--time cylinder $Q$.
Due to the choice of the upwind-in-time numerical flux 
$\widehat\psi_{hp}=\psi^-_{hp}$ on $\Fspa$, this system can be decomposed as a 
sequence of $N$ smaller linear systems: each of them arises 
from solving sequentially for $\psi_{hp}$ in each time-slab $\Omega \times 
[t_{n - 1}, t_n]$, and using
the trace of the solution from the previous slab as initial datum. 

Furthermore, by choosing the same space mesh $\Thxn$ and time step 
$\tau = t_n - t_{n - 1}$ for all $n = 1, \ldots, N$, we can apply a time 
translation for each time-slab $\Omega \times [t_{n - 1}, t_n]$ in the 
definition of the basis functions \cref{EQN::BASIS-FUNCTIONS}, as 
$\phi_{\ell}(\bx, t) := 
\exp[i(k_\ell \bd_\ell^\top \bx - (k_{\ell}^2 + V\big|_{K})(t - t_{n -1}))]
$. This makes the matrices of 
the linear systems for all the time-slabs to be the same, which represents 
a substantial reduction in the computational cost of the method.
To solve these systems we 
perform the LU factorization of such matrix once
using the Matlab's function \texttt{lu} with scaling and row--column 
permutations, which produces sparser and stable factorizations; then we 
solve for each time slab applying forward and backward substitutions.

As it is usual for plane-wave 
approximations \cite[\S4.3]{Hiptmair_Moiola_Perugia_2016}, 
the time-stepping matrix is ill-conditioned. We observe $\calO(h^{-(2p + 
1)})$ growth of the 2-condition number $\kappa_2$ under uniform space--time mesh 
refinement for both the ($1+1$)- and the ($2+1$)-dimensional 
cases. A theoretical study of this matter as well as suitable preconditioning 
(or condition-number reduction)
techniques for this method is highly relevant and will be the subject of 
future work.

The stabilization parameters $\alpha$ and $\beta$ are taken as in 
\Cref{THM::ERROR-ESTIMATE} 
with $h_{F_{\bx}} = \min\{h_{K_x^1}$, $h_{K^2_\bx}\}$ for the  faces in $\Ftime$.

The integrals in the assembly of the Galerkin matrix and load vectors are 
computed with Gauss--Legendre quadratures (combined with the Duffy transform 
for the integrals over triangles).
On polytopic meshes, thanks to the choice of exponential basis functions 
\eqref{EQN::BASIS-FUNCTIONS}, closed formulas for all the integrals appearing in 
the matrix assembly could be written, following the ideas in 
\cite[\S4.1]{Hiptmair_Moiola_Perugia_2016}{: the implementation of these formulas is non-trivial, particularly in higher dimensions, but it could considerably speed up the computations}.
We recall that the Trefftz-DG formulation \eqref{EQN::VARIATIONAL-TREFFTZ-DG} 
does not involve ($d+1$)-dimensional integrals over mesh elements but only on 
the $d$-dimensional element faces.

\subsection{Square potential well in \texorpdfstring{$1+1$}{1+1} dimensions}
\label{s:NumExp1D}
Let us consider the (1+1)-dimensional Schr\"odinger equation 
\cref{EQN::SCHRODINGER-EQUATION} on $Q = (-2, 2) \times (0, 1)$ with 
homogeneous Dirichlet boundary conditions and the following square-well potential:
\begin{equation}
\label{EQN::SQUARE-WELL}
V(x) = \left\{  
\begin{tabular}{ll}
	$0$, & $x\in (-1, 1),$\\
	$V_*$, & $x \in (-2, 2)\ \setminus\ (-1, 1),$
\end{tabular}
\right.	
\end{equation}
for some $V_* > 0$.
The initial condition is taken as an eigenfunction (bound state) of 
$-\partial_x^2+V$ on $(-2,2)$:
\begin{equation*}
\psi_0(x) = \left\{
\begin{tabular}{ll}
$\cos(k_*x)$, & $x \in (-1, 1),$ \\
$\frac{\cos(k_*)}{\sinh(\sqrt{V_* - k_*^2})} \sinh(\sqrt{V_* - k_*^2}(2 - |x|))$, & $x \in (-2, 2) \ \setminus\ (-1, 1)$,
\end{tabular}
\right.
\end{equation*}
where $k_*$ is a real root of the function $f( k) := \sqrt{V_* - k^2} - 
k\tan( k) \tanh(\sqrt{V_* - k^2} )$.
The solution of the corresponding initial boundary value problem 
\eqref{EQN::SCHRODINGER-EQUATION} is $\psi(x, t) = \psi_0(x)\exp(-ik^2 t)$.
For each $V_*$ there is a finite number of such values $k_*$: in the
numerical experiments below we take the largest one, corresponding 
to faster oscillations in space and time.
In \Cref{FIG::PLOT-ZERO-FUNC} we present the plot the function 
$f(x)$ for 
$V_* = 20$ and $V_* = 50$ with the values of $k_*$ used in this experiment.

\begin{figure}[htb]
\centering
\subfloat[$V_* = 20$]{\label{FIG::F-V20}
\includegraphics[width=.49\textwidth, clip]{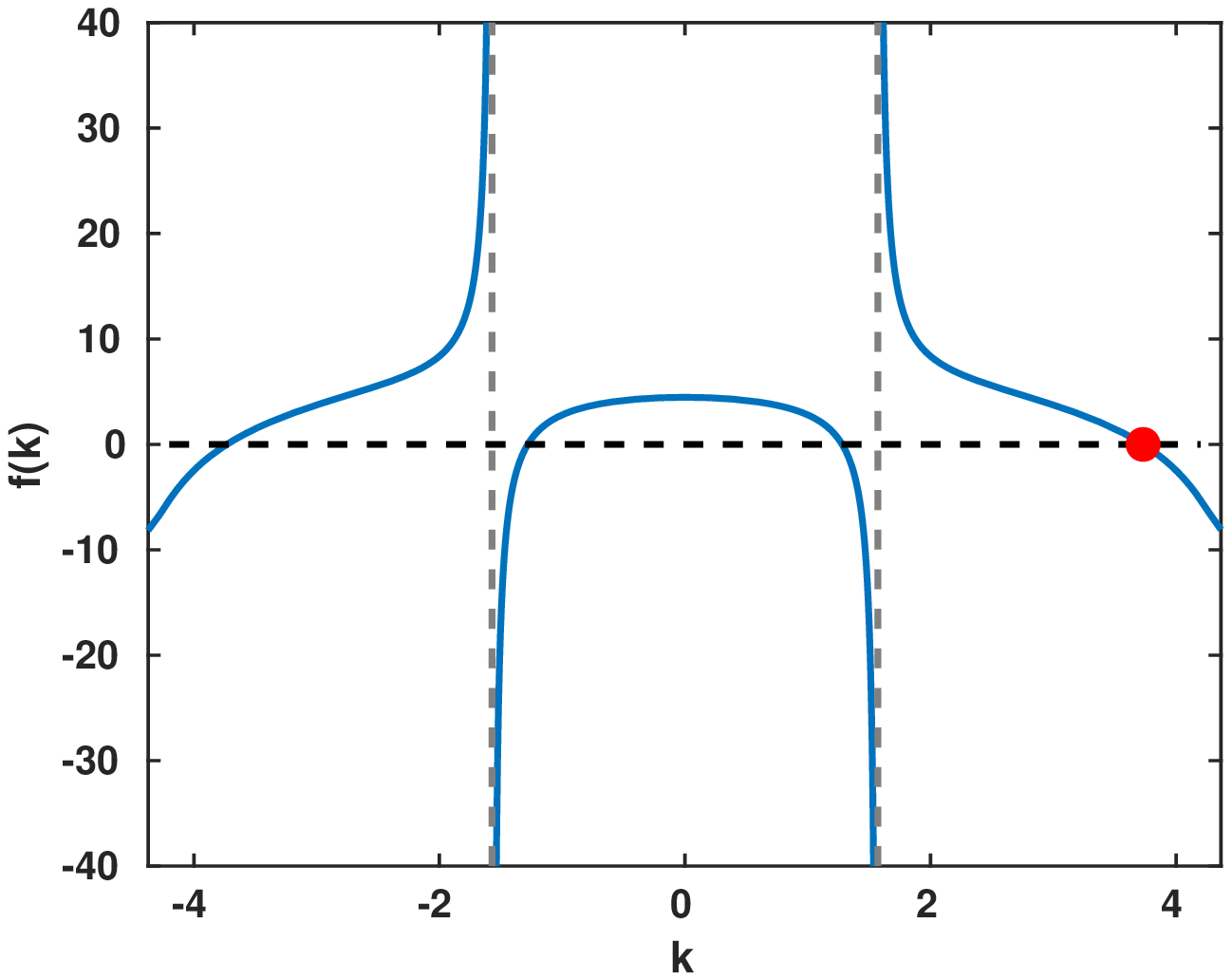}}
\subfloat[$V_* = 50$]{\label{FIG::F-V50}
\includegraphics[width=.49\textwidth, clip, trim=10 0 40 20]{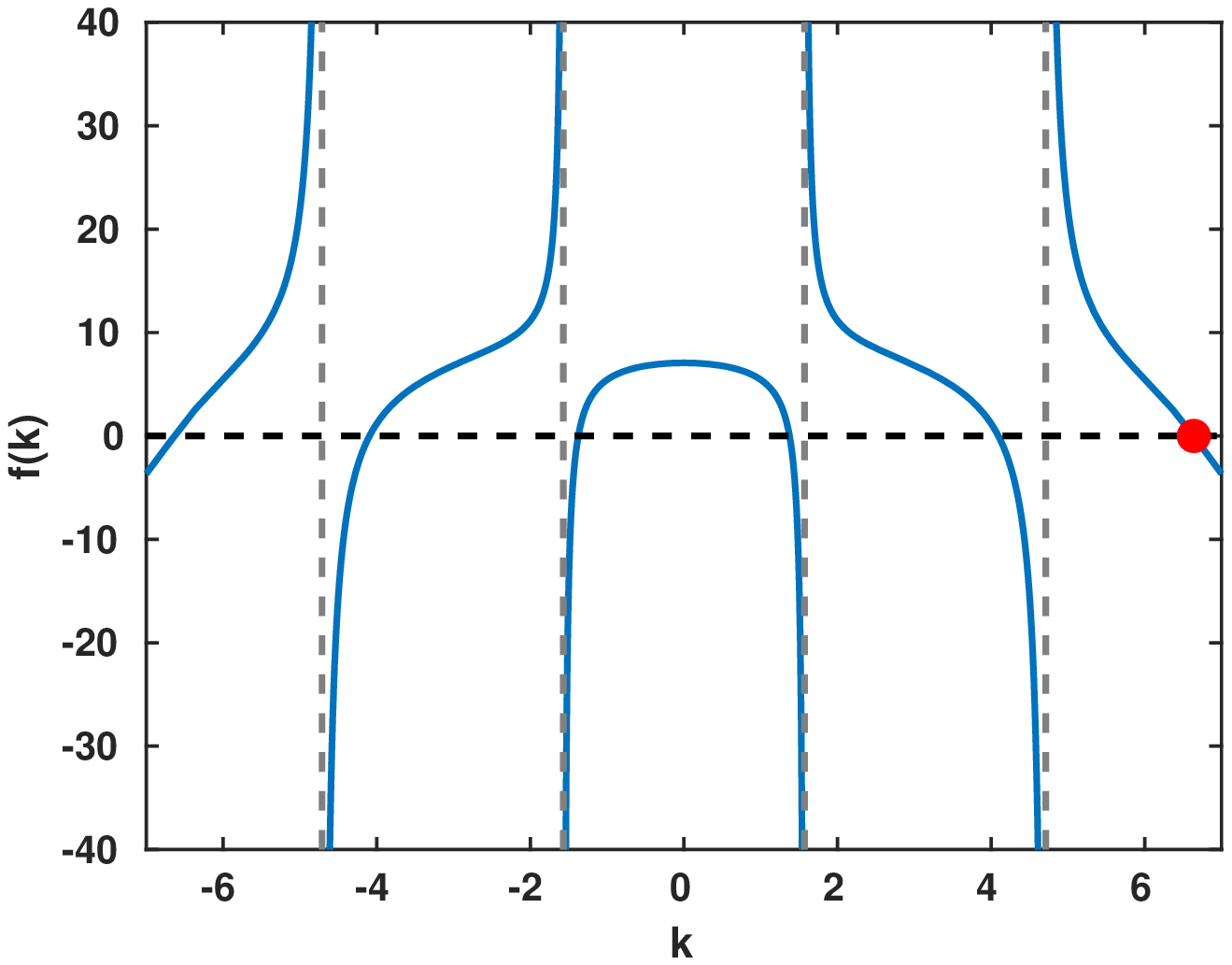}}
\caption{Plot of $f(k)$ for different values of $V_*$.
The red dots are the values of $k$ taken in the numerical experiments: $k_ * 
\approx 3.7319\ (V_* = 20)$ and $k_ * \approx 6.6394\ (V_* = 50)$ \label{FIG::PLOT-ZERO-FUNC}.}
\end{figure}

From the definition \eqref{EQN::BASIS-1+1} of the Trefftz basis, we observe 
that if the parameters $k_{\ell}$ are chosen very close to one another then the 
corresponding basis functions approach mutual linear dependence.
The consequence is that the Trefftz-DG method becomes more and more ill-conditioned.
This is confirmed by the Vandermonde-like matrix $\bV$ in \cref{EQN::VANDERMONDE-MATRIX}, which is singular in the limit $|k_{\ell_1}-k_{\ell_2}|\to0$.
In the experiments we take $2p+1$ equally spaced values 
$k_\ell\in\{-p, -(p - 1), \ldots, 0, \ldots, p - 1, p\}$.

In \Cref{FIG::PLOT-V20,FIG::PLOT-V50} we plot the Trefftz-DG numerical 
approximations $\psi_{hp}$ obtained for $p = 3$ on the finest mesh 
(described below) with $V_* = 20$ and $V_* = 50$ respectively.
It can be observed that for increasing $V_*$ the solution oscillates more 
with respect to $x$ in $(-1, 1)$, while it decays monotonically to 0 in $(-2, 
2) \setminus (-1, 1)$.

\begin{figure}[htb]
\subfloat[$\REL{\uhp}$ for $V_* = 20$]{\label{FIG::PLOT-V20}
\includegraphics[width=.49\textwidth, clip, trim=0 0 20 10]{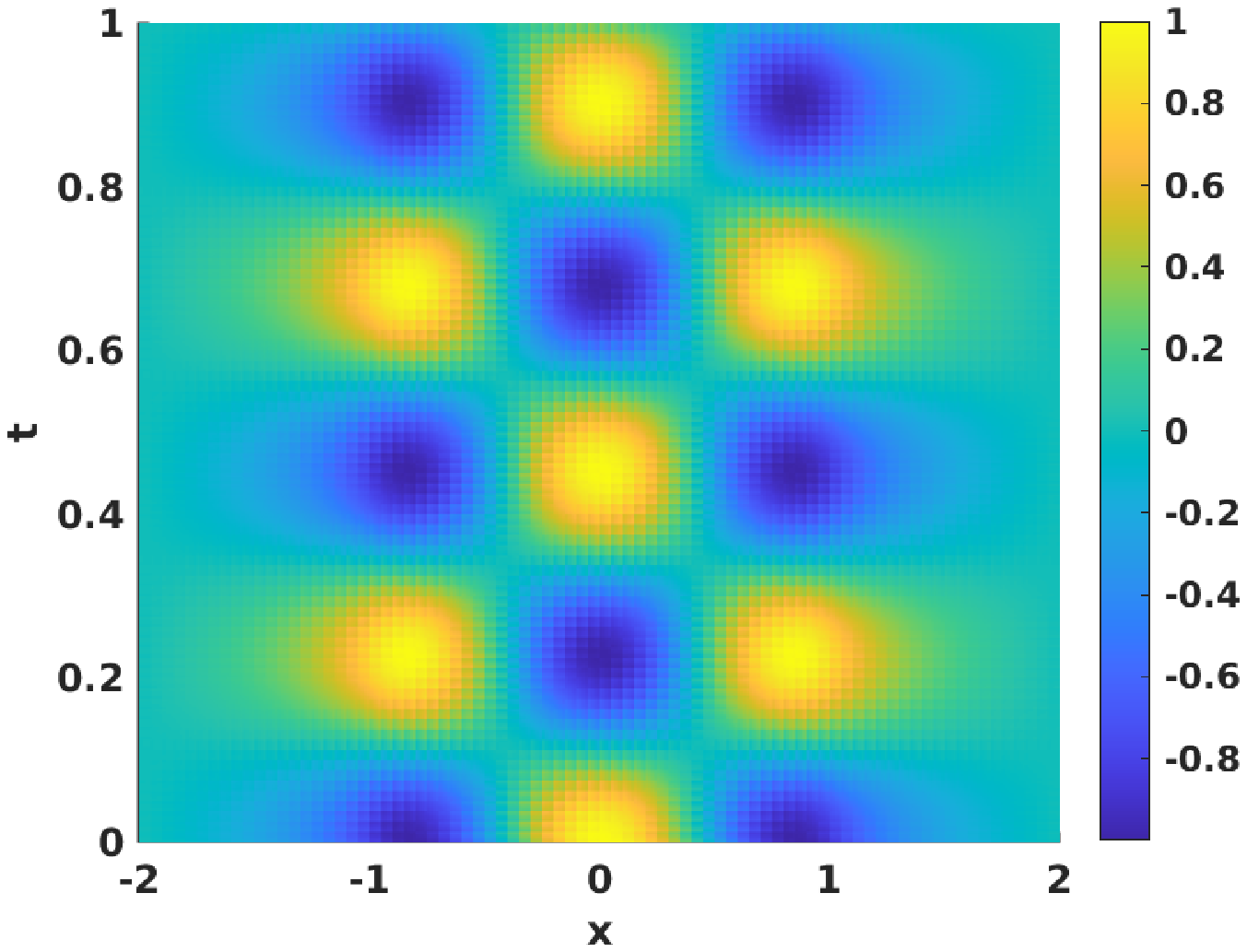}}
\subfloat[$\REL{\uhp}$ for $V_* = 50$]{\label{FIG::PLOT-V50}
\includegraphics[width=.49\textwidth, clip, trim=0 0 20 10]{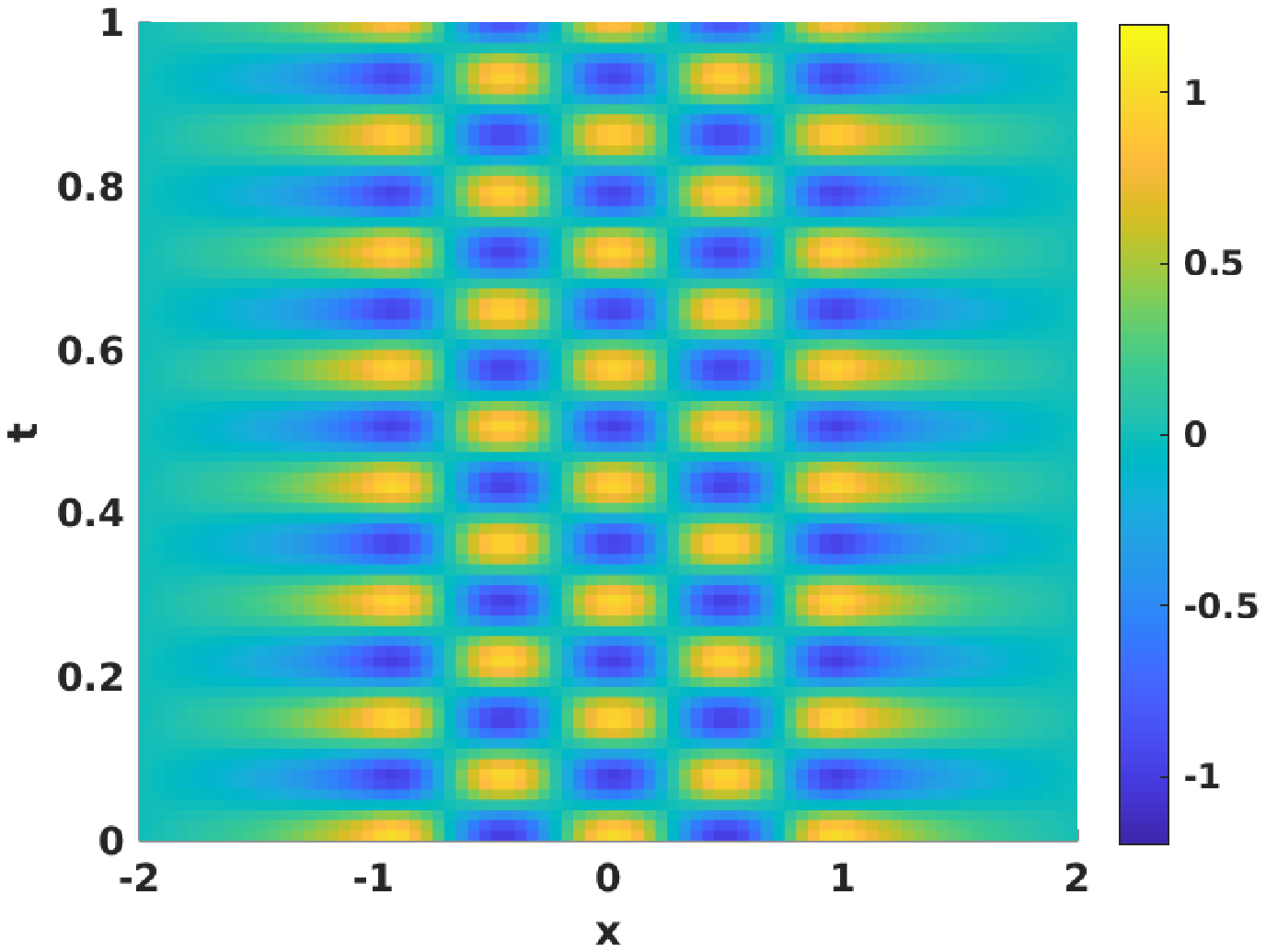}}
\caption{Trefftz-DG approximation $\psi_{hp}$ in the space--time cylinder $Q$ 
for the ($1+1$)-dimensional square-well potential problem \eqref{EQN::SQUARE-WELL} 
computed with $p = 3$.}
\end{figure}

In \Cref{FIG::1D-DG-ERROR} we 
plot the DG norm of the Galerkin error obtained for 
$V_* = 20$ and a sequence of space--time, uniform, Cartesian meshes with
$h_\bx = 0.2,\ 0.1,\ 0.0667,\ 0.05$, 
$h_t = 0.25h_\bx$ and $p=1,2,3$.

\begin{figure}[!ht]
\subfloat[Error in DG norm]{\label{FIG::1D-DG-ERROR}
\includegraphics[width=.49\textwidth, clip]{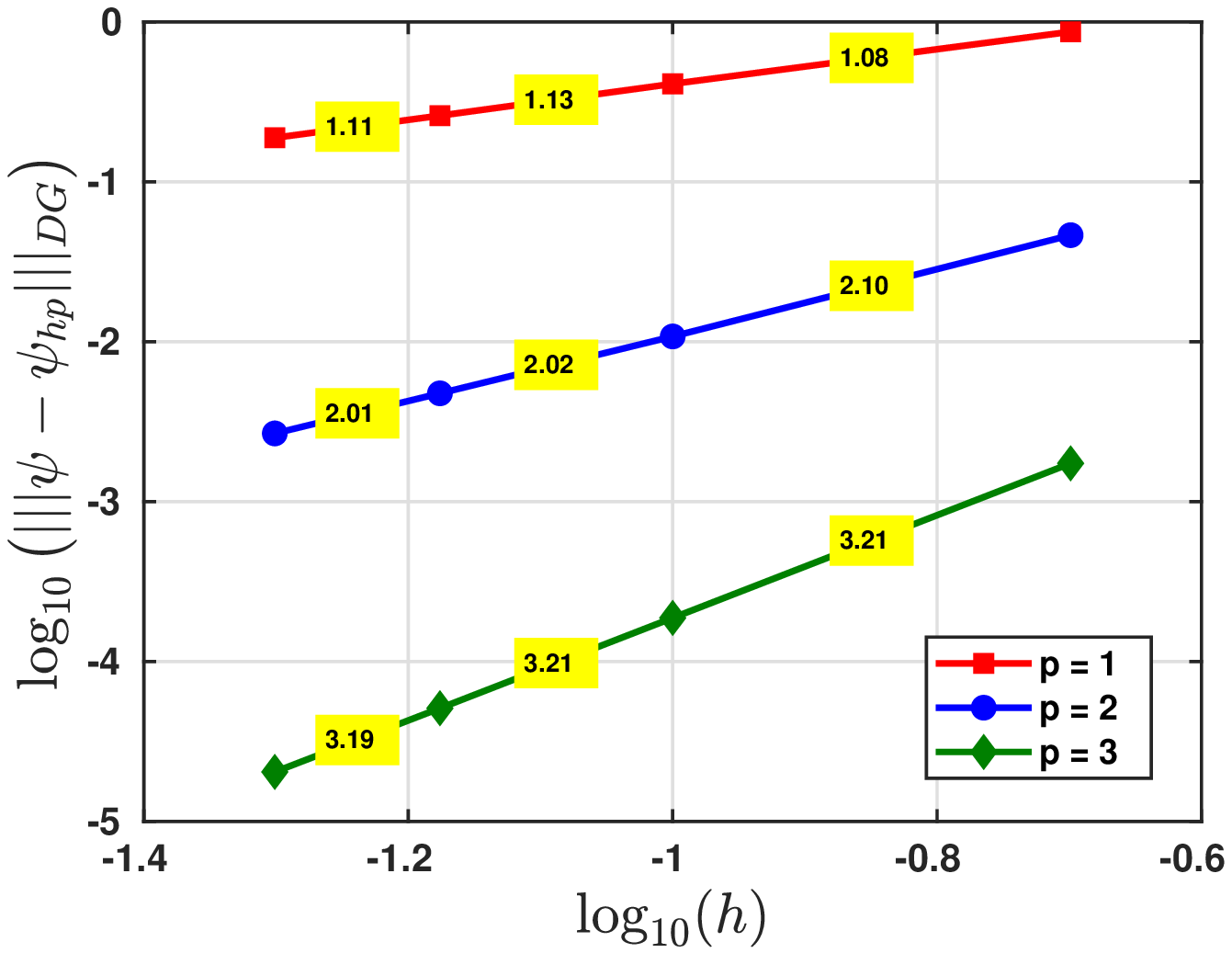}}
\subfloat[Error in $L_2$ norm at $T = 1$]{\label{FIG::1D-L2-ERROR}
\includegraphics[width=.49\textwidth, clip]{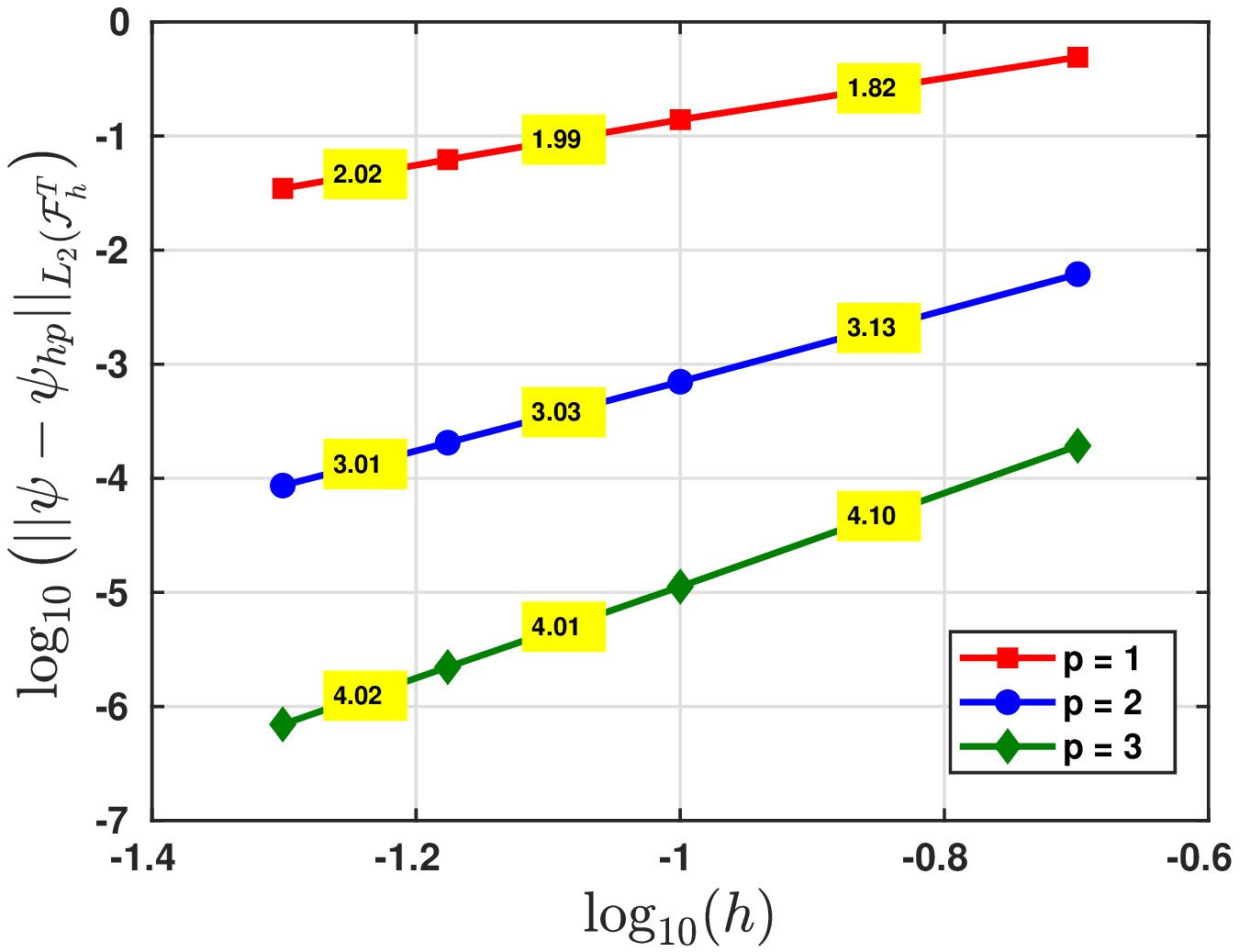}}
\caption{Trefftz-DG error for the $(1 + 1)$-dimensional problem with square 
well potential \cref{EQN::SQUARE-WELL} with $V_* = 20$. 
The numbers in the yellow rectangles are the empirical algebraic convergence 
rates in $h$.
\label{FIG::ERROR-1D}}
\end{figure}

Since we have homogeneous Dirichlet boundary conditions, the continuous model 
preserves the energy functional $\calE(t,\psi)$, recall \cref{Rem:Energy}.
In \cref{FIG::ENERGY-ERROR-d1} 
we show the time-evolution of the energy 
error for the Trefftz-DG approximation for the finest mesh , which is 
smaller for larger
$p$, as expected.
Moreover, in \cref{FIG::ENERGY-LOSS-d1} we numerically observe that $\calE_{loss}$ converges to zero as 
$\calO(h^{2p})$, as it can be proved combining \cref{THM::ERROR-ESTIMATE} 
and \cref{Rem:Energy}.

\begin{figure}
\centering
\subfloat[Energy error evolution ]{\label{FIG::ENERGY-ERROR-d1}
\includegraphics[width=.49\textwidth, clip]{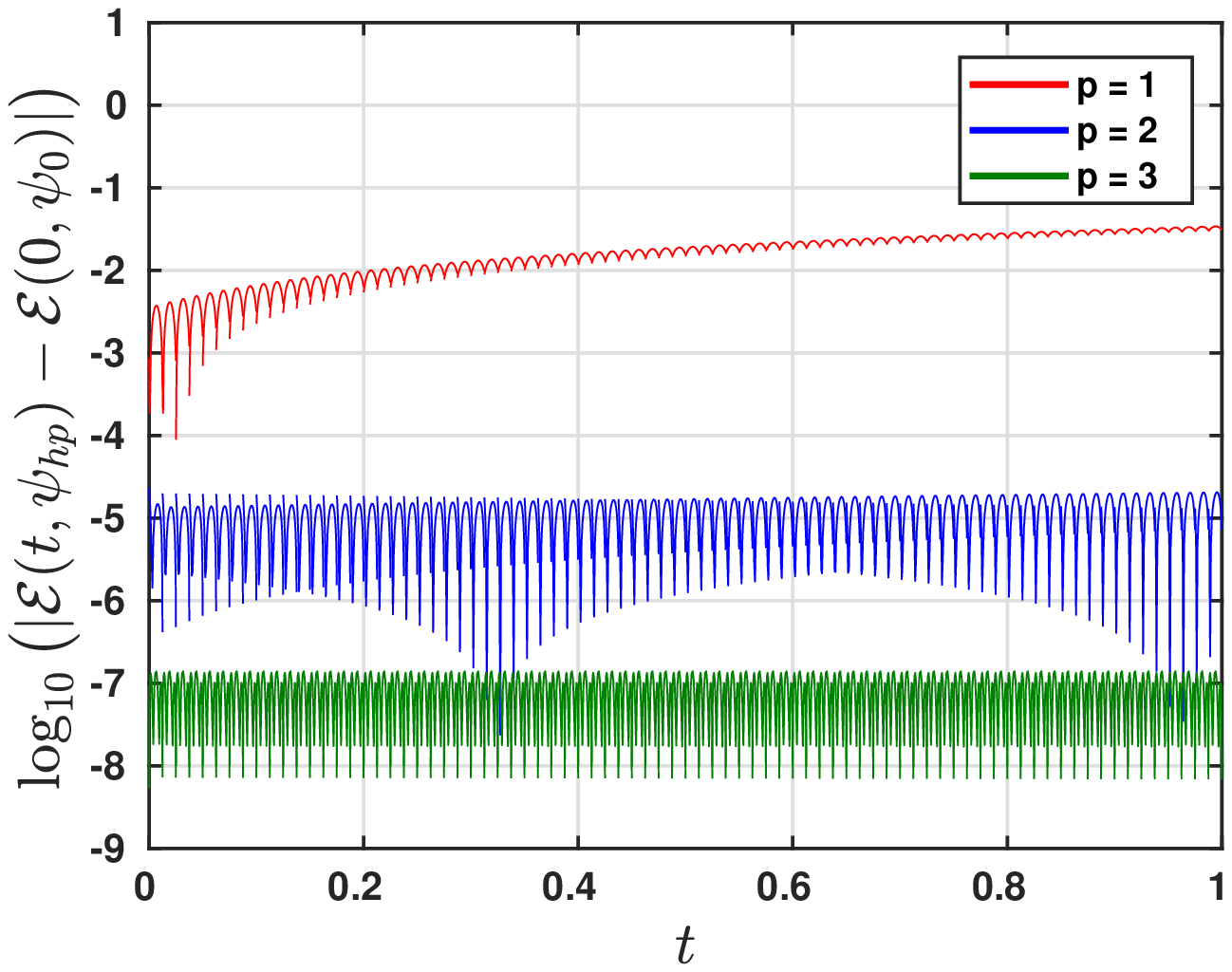}}
\subfloat[Energy loss at $T = 1$ ]{\label{FIG::ENERGY-LOSS-d1}
\includegraphics[width=.49\textwidth, clip]{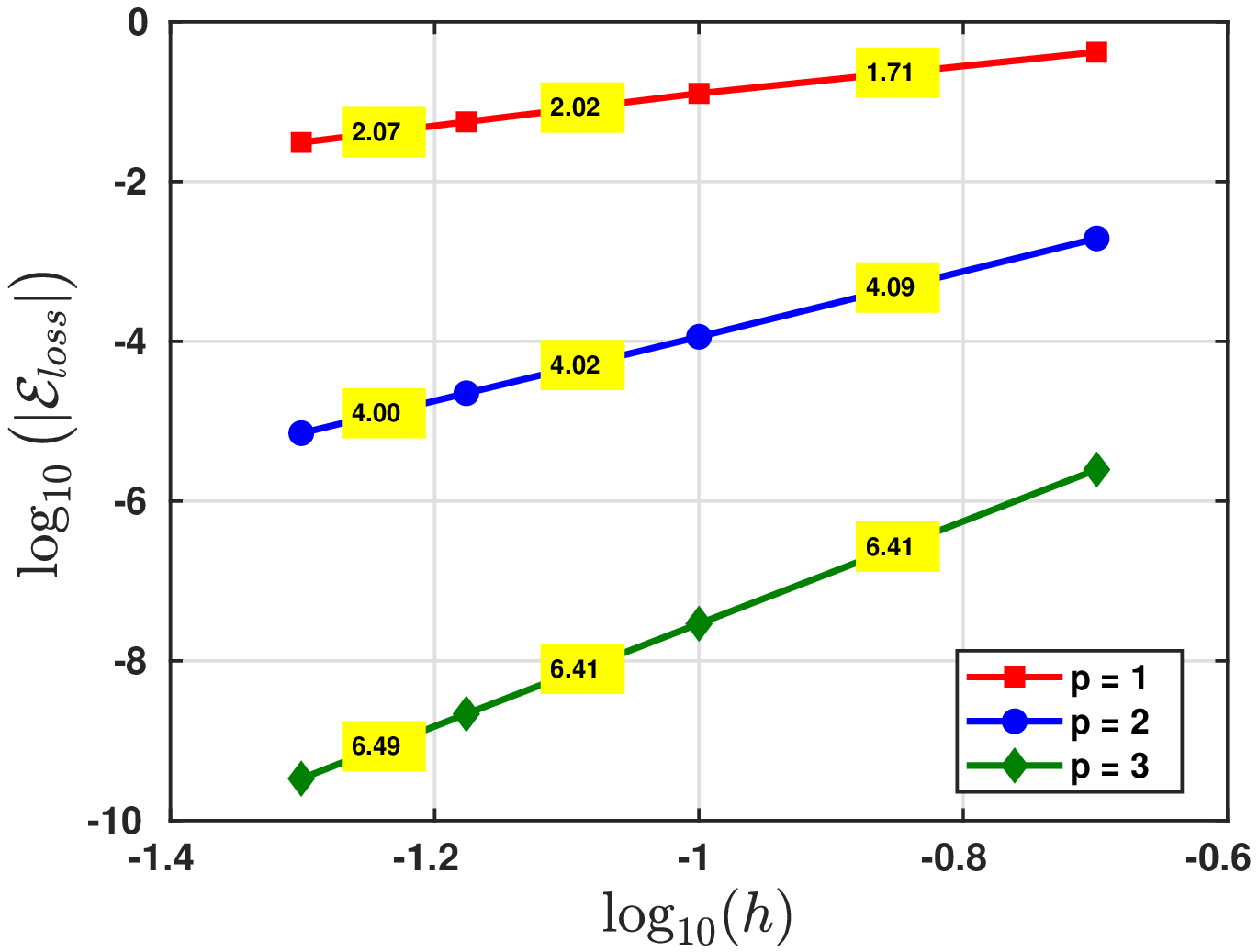}}
\caption{Time-evolution of the energy error and dependence on $h$ and $p$ for 
the problem with square-well potential and $V_* = 20$.}
\end{figure}

In order to see the effect of the choice of the parameters $k_\ell$, we first 
note that in this experiment we know the time frequency of the exact solution, 
which is $\omega = k_*^2$. Therefore it is natural to 
expect the approximation to be better if our basis functions oscillates at the same time frequency.
To numerically illustrate this, in 
\Cref{FIG::1D-DG-ERROR-V50,FIG::1D-DG-ERROR-V100} we show
that the convergence rates clearly degrade for $p = 1$, $V_* = 50$ and 
$V_*= 100$, and our previous choice of the parameters $k_\ell$,
as the time frequencies of the basis functions are too far from those of the 
exact solution. On the contrary, by taking the parameters $k_\ell$ as $\{ 
-k_*, 0, k_*\}$ we 
recover the expected rates. This clearly suggests that a sensible tuning of 
 the basis function parameters can significantly improve the accuracy of 
the method.

\begin{figure}[!ht]
\subfloat[$V_* = 50$]{\label{FIG::1D-DG-ERROR-V50}
\includegraphics[width=.49\textwidth, clip, trim=0 0 30 10]{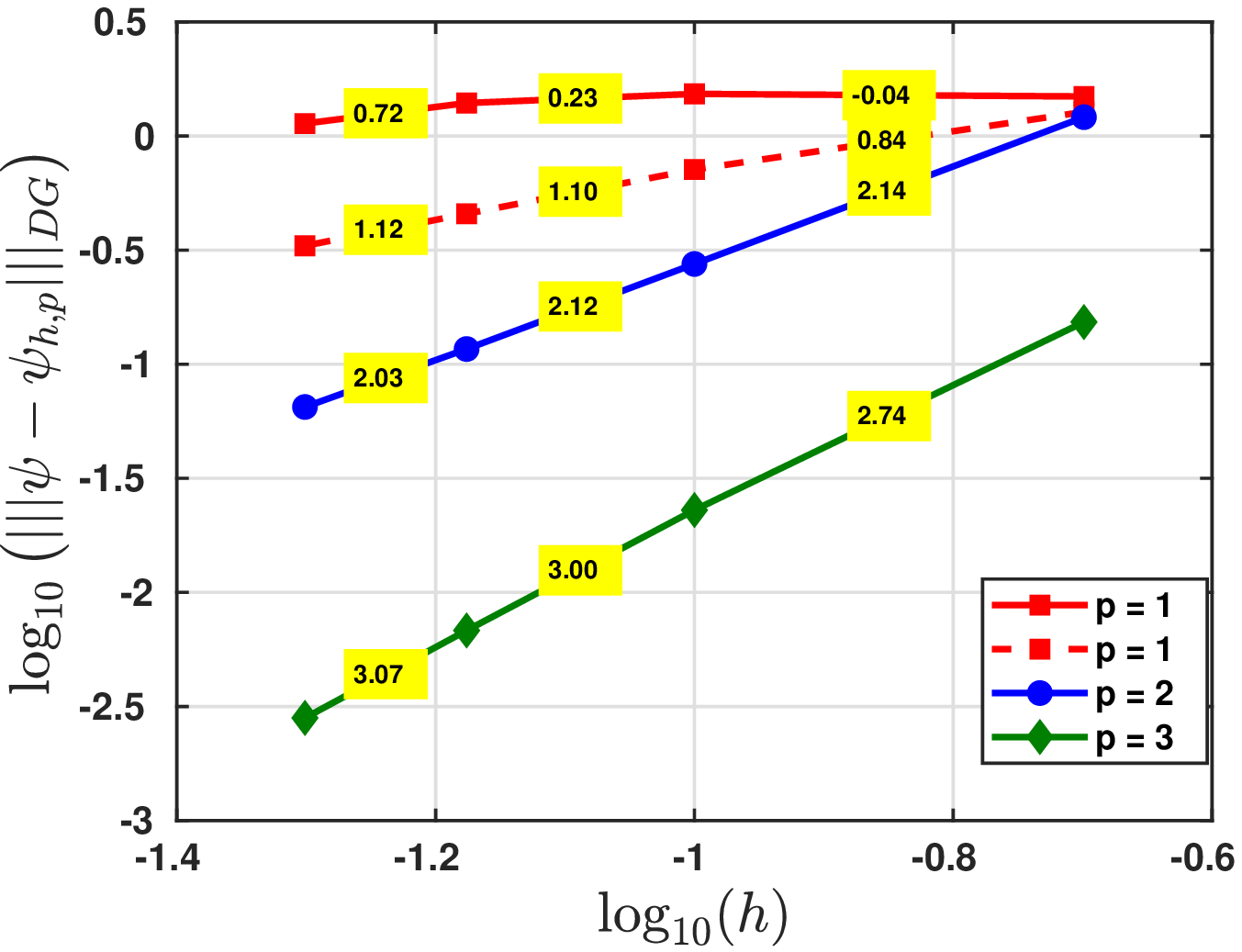}}
\subfloat[$V_* = 100$]{\label{FIG::1D-DG-ERROR-V100}
\includegraphics[width=.49\textwidth, clip, trim=0 0 30 10]{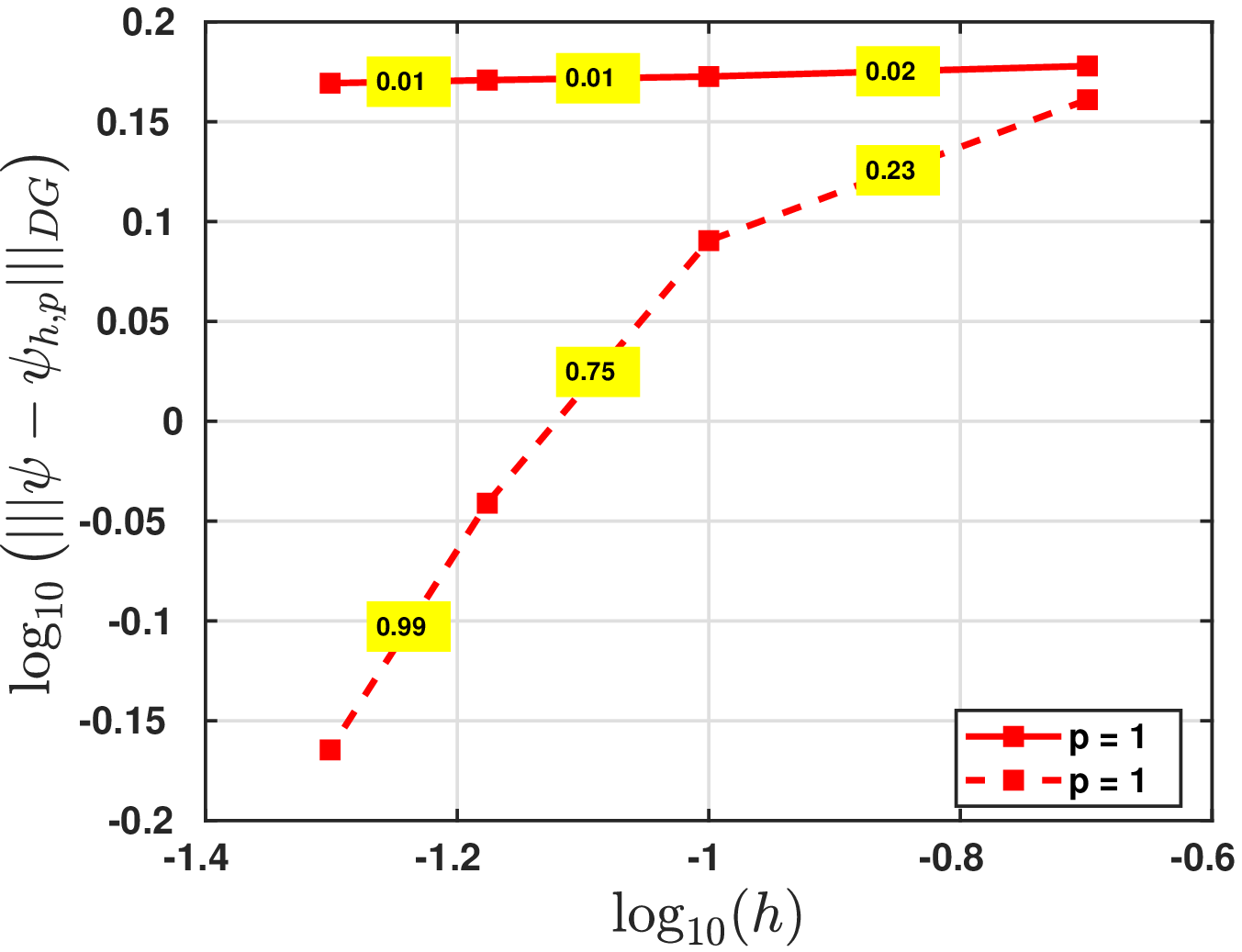}}
\caption{Trefftz-DG error measured in DG norm for the (1 
+ 1) dimensional problem with square-well potential \cref{EQN::SQUARE-WELL} 
with 
$V_* = 50\ (k_* \approx 6.6394)$ and $V_* = 100\ (k_* \approx 9.6812)$, 
and for 
$k_\ell\in\{-p,\ldots,p\}$ (continuous line), which is the same choice of the 
previous plots, and $k_\ell\in\{0,\pm k_*\}$ (dashed line).
\label{FIG::ERROR-1D-V50}}
\end{figure}

\subsection{(\texorpdfstring{$2 + 1$}{2+1})-dimensional transient Gaussian distribution}

We consider the linear Schr\"odinger equation 
\eqref{EQN::SCHRODINGER-EQUATION} with zero potential $V = 0$ on $Q = 
\Omega\times(0,2)$, with $\Omega = (-2, 4) \times (-2.5, 2.5)$.
Following \cite{Antonie_Besse_Mouysset_2004}, the initial and boundary 
conditions are chosen such that the exact solution is
$$
\psi(x, y, t) =  \frac{i}{i - 4t} \ee^{-\frac{i}{i - 4t} (x^2 + y^2 + ix + it)}.
$$
The basis function parameters are chosen as $p+1$ equally spaced space 
wavenumbers $k_{m} = 1, 2, \ldots, p + 1$, and the equally spaced angles 
$\theta_{m,\lambda} = \frac{2\pi(\lambda-1)}{2m + 1}$
in $(0, 2\pi]$, $\lambda = 1, \ldots, 2m+1$, as in \cref{fig:2Dbasis}.

In \Cref{FIG::2D-DG-ERROR} we show the convergence rates of the 
Trefftz-DG approximation for a set of structured triangular meshes generated 
by halving each rectangle in a Cartesian partition of $\Omega$ with equal 
number of divisions (20, 40, 60, 80) in both $x$ and $y$ directions.
In this experiment $h_t \approx 0.5 h_\bx$. 
The $\calO(h^p)$
convergence rates obtained are in agreement with \cref{THM::ERROR-ESTIMATE}.
Similarly to the ($1+1$)-dimensional example in the previous section, in 
\cref{FIG::2D-L2-ERROR} we observe $\calO(h^{p+1})$ convergence rates
in the $L_2\OO$ norm at the final time. Not shown here, similar results were 
obtained for rectangular meshes in space.

In \cref{FIG::ERROR-2D-DOF} we study the $p$-convergence of the method: for the 
two coarsest space meshes with $h_t \approx h_\bx/8$, and for $p\in\{1,2,3,4\}$, 
the DG norm of the error is plotted against the total number of degrees of 
freedom in $Q$.
We observe $\calO(\ee^{-b\sqrt{\#DOF}})$ convergence when the ``local degree'' $p$ is raised.
This is in strong contrast with what one might expect from a polynomial method: 
in that case only the slower rate $\calO(\ee^{-b\sqrt[3]{\#DOF}})$ can be 
achieved (recall \cref{Rem:LessDOFs}).
The Trefftz-DG approximations for $p = 3$, at the initial and final times, are 
shown in \Cref{FIG::PLOT-2D}.

\begin{figure}[!ht]
\subfloat[Error in DG norm]{\label{FIG::2D-DG-ERROR}
\includegraphics[width=.49\textwidth,clip]{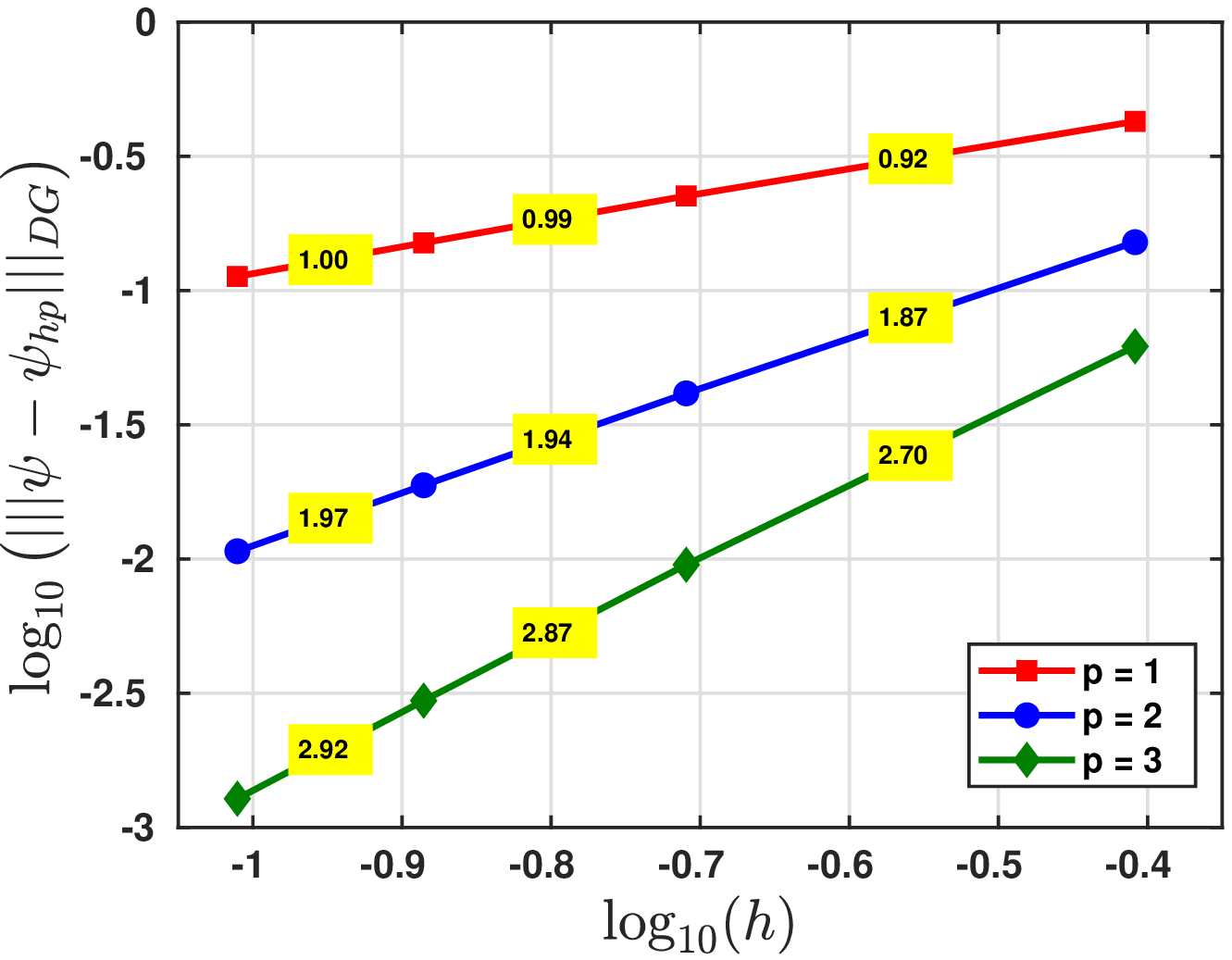}}
\subfloat[Error in $L_2$ norm at $T = 2$]{\label{FIG::2D-L2-ERROR}
\includegraphics[width=.49\textwidth,clip]{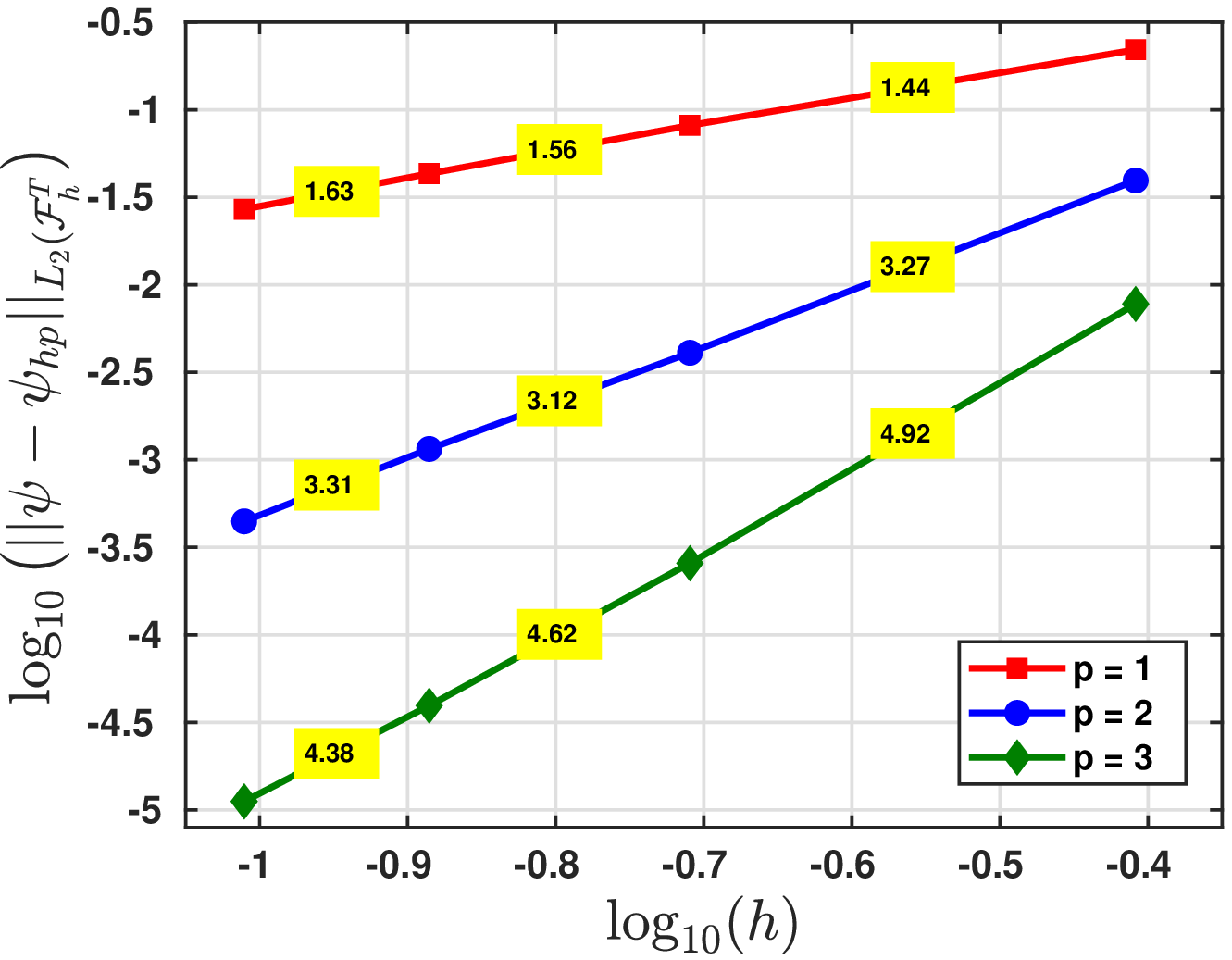}}
\caption{Trefftz-DG error  for the ($2 + 1$)-dimensional transient Gaussian problem. \label{FIG::ERROR-2D}}
\end{figure}

\begin{figure}[!ht]
\subfloat[Mesh 1, $h_\bx\approx 0.3905$]{\label{FIG::2D-DG-ERROR-DOF}
\includegraphics[width=.49\textwidth,clip]{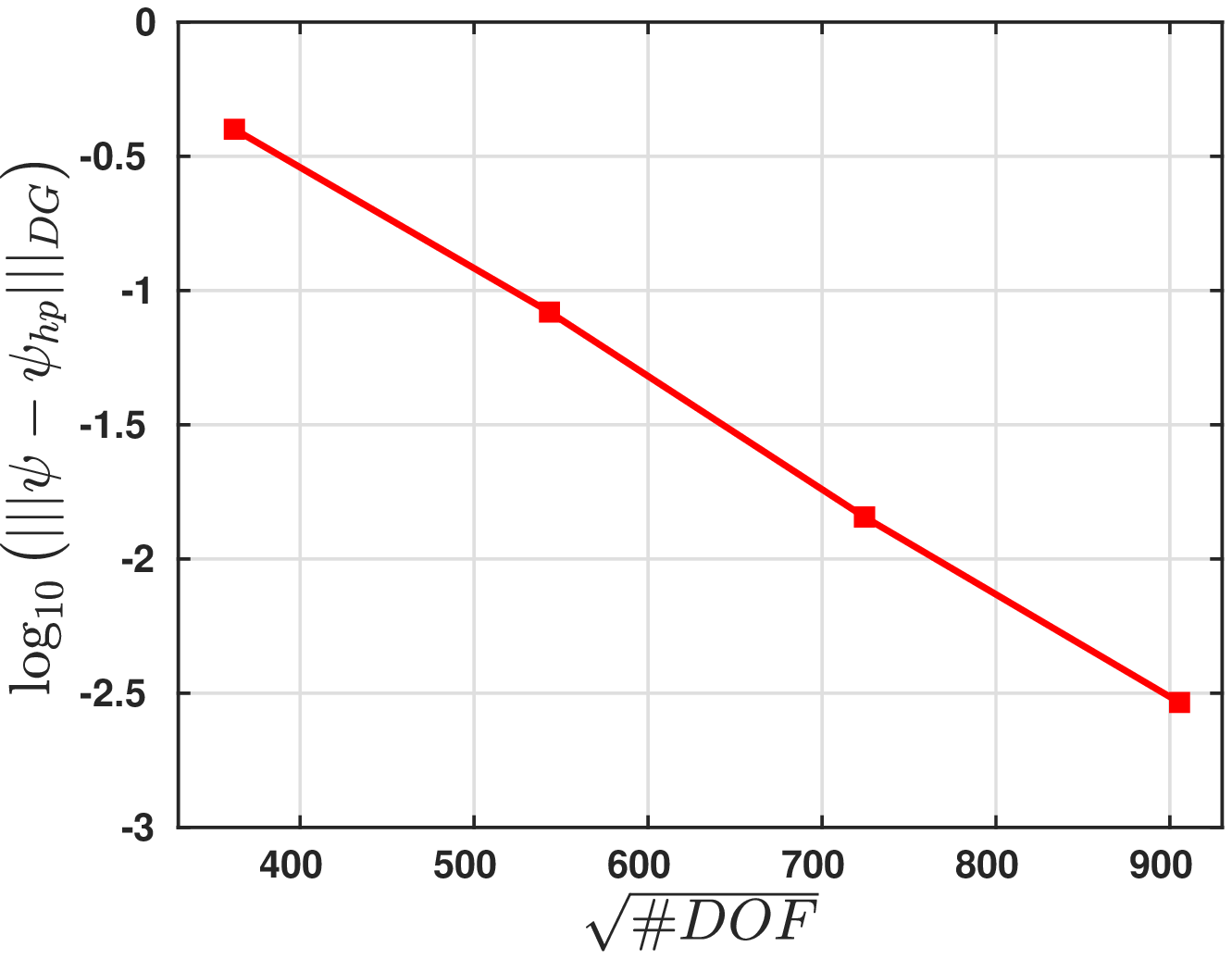}}
\subfloat[Mesh 2, $h_\bx \approx 0.1953$]{\label{FIG::2D-L2-ERROR-DOF}
\includegraphics[width=.49\textwidth,clip]{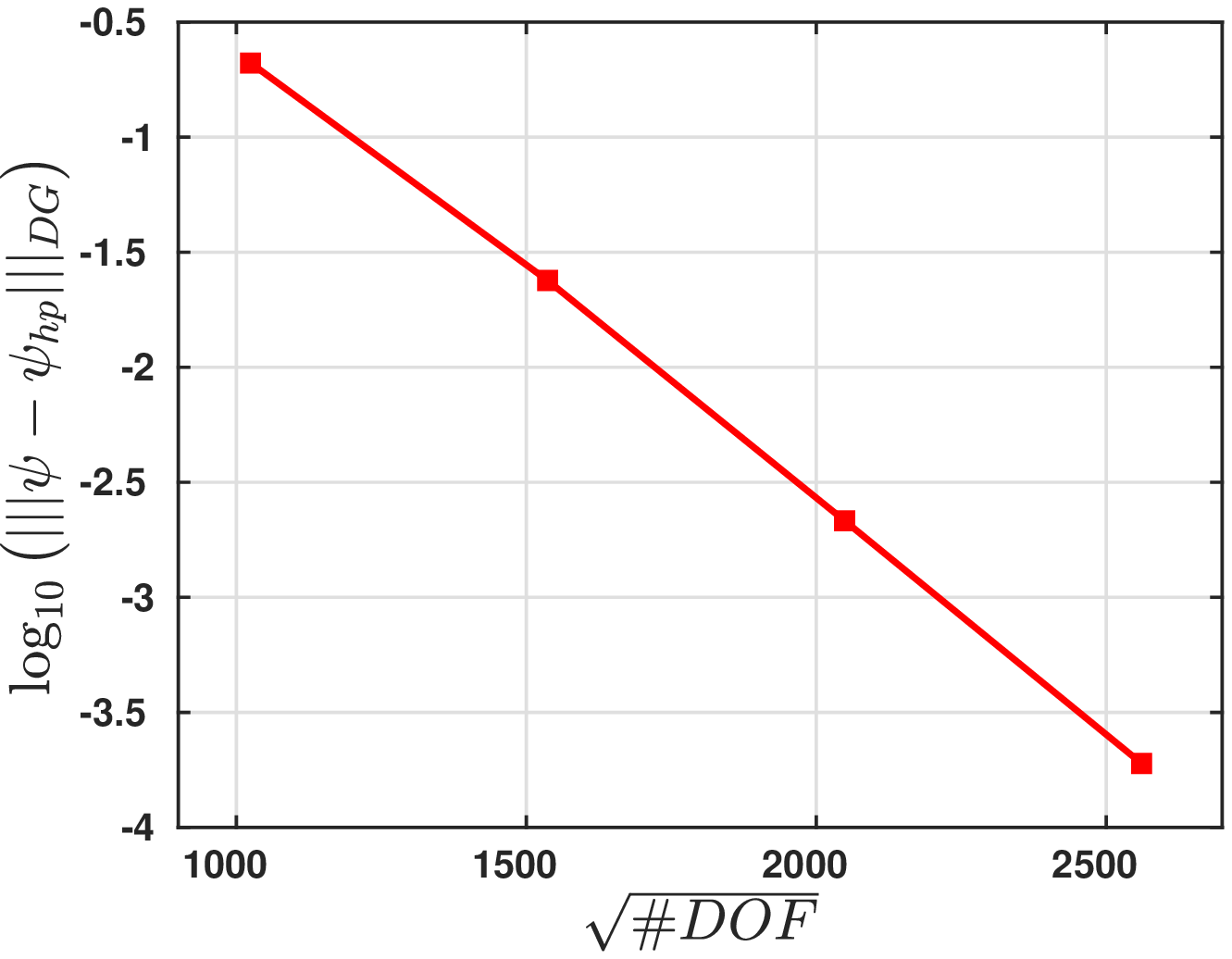}}
\caption{$p$-convergence of the Trefftz-DG error against the squared 
root of the total number of degrees of freedom for the ($2 + 1$)-dimensional 
transient Gaussian problem. \label{FIG::ERROR-2D-DOF}}
\end{figure}

\begin{figure}[!ht]
\subfloat[$t = 0$]{\label{FIG::2D-PLOT-0}
\includegraphics[width=.49\textwidth,clip,trim=0 0 20 15]{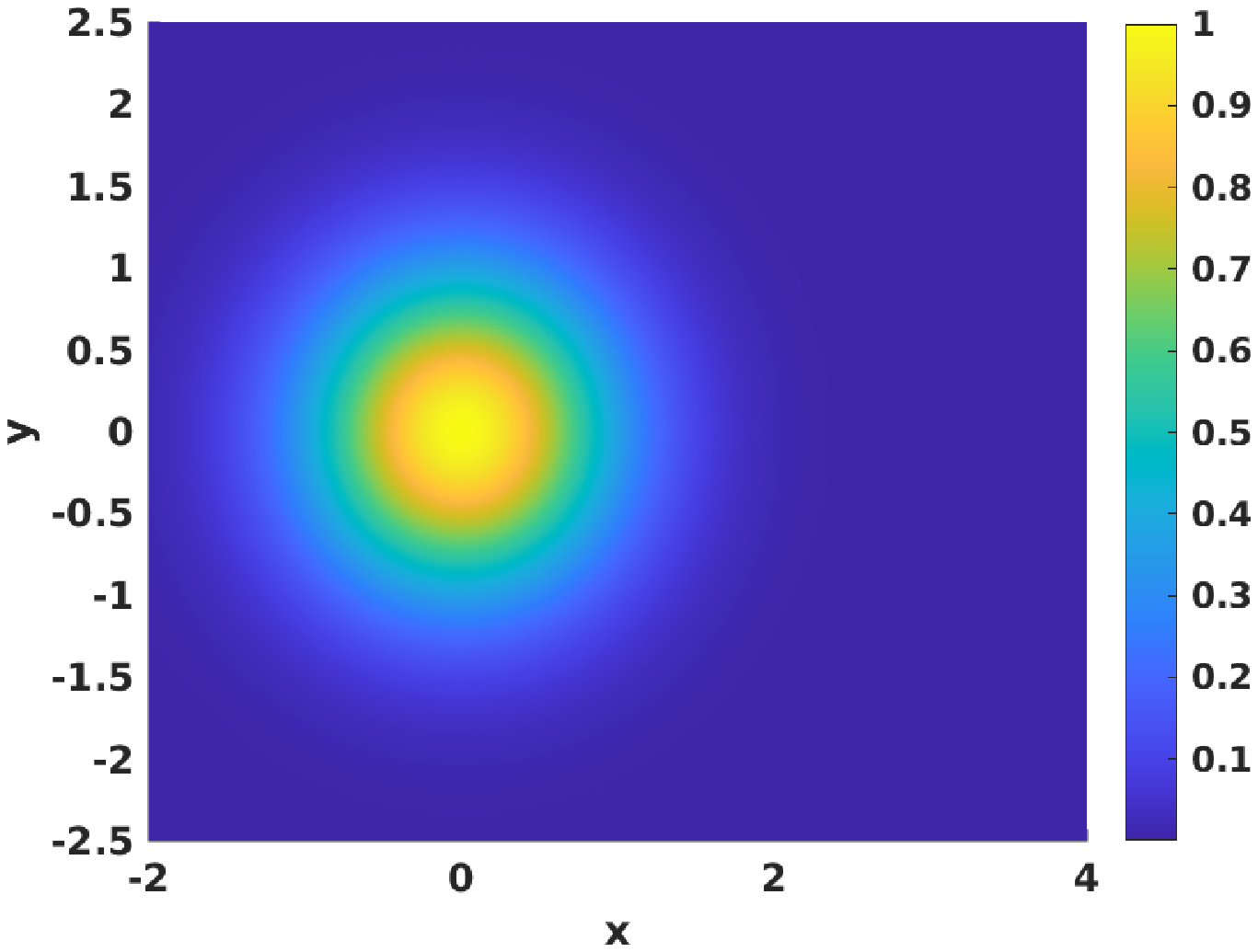}}
\subfloat[$t = 2$]{\label{FIG::2D-PLOT-F}
\includegraphics[width=.49\textwidth,clip,trim=0 0 20 15]{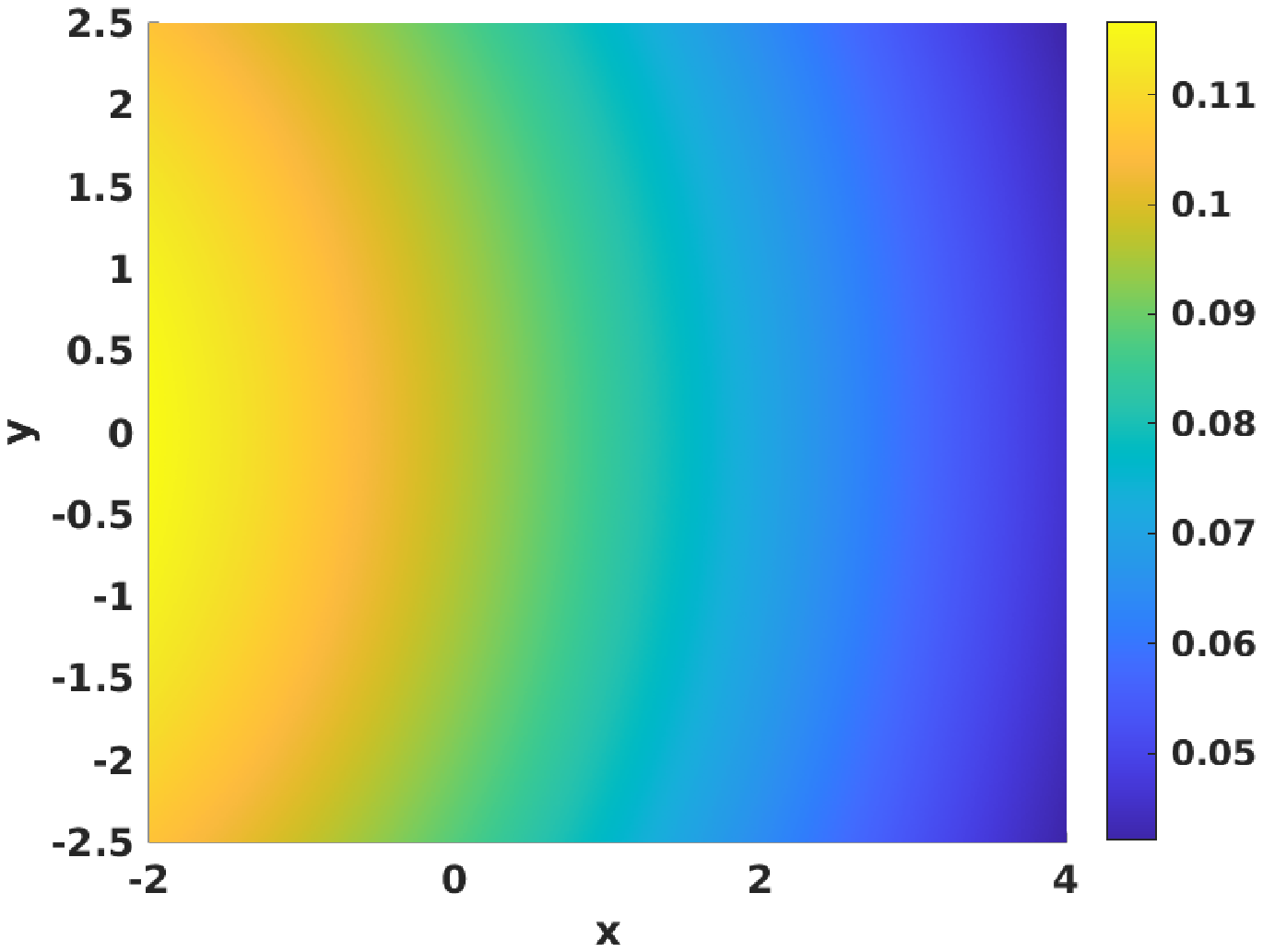}}
\caption{Trefftz-DG approximation of the ($2 + 1$)-dimensional transient Gaussian. \label{FIG::PLOT-2D}}
\end{figure}

\section{Concluding remarks\label{SECT::CONCLUSIONS}}		

We have introduced a Trefftz DG method for the approximation of the 
time-dependent, linear Schr\"odinger equation.
We have analysed its well-posedness, stability and $h$-convergence properties.
As this is the first description of such a numerical scheme,  
several extensions and improvements of the method and its analysis might be addressed.
We list here a few possible future research directions.
\begin{itemize}
\item The extension of the $h$-convergence bounds to space dimensions higher 
than~2.

\item The proof of optimal error estimates in mesh-independent norms such as $L^2(Q)$ (cf.\ the duality approach used in \cite[\S5.4]{Moiola_Perugia_2018} for the wave equation).

\item The analysis of the method in locally refined space--time meshes, as in \cite{Moiola_Perugia_2018}.

\item The $p$-convergence analysis, i.e.\ the proof of convergence rates for 
sequences of discrete spaces obtained by local enrichment on a fixed mesh (the 
Trefftz equivalent of increasing the local polynomial degrees).
Only the best-approximation bounds are missing, as the quasi-optimality bound \eqref{EQN::QUASI-OPTIMALITY} is independent of the discrete space.
This is a very challenging task, given the nature of the Trefftz basis functions.
In fact, this was accomplished for the Helmholtz equation 
\cite{Moiola_Hiptmair_Perugia_2011} but not yet for the wave equation.

\item The ``sparsification'' of the scheme, i.e.\ the combination with sparse-grid techniques in space--time to improve its efficiency, e.g.\ along the lines of \cite{Bansal_Moiola_Perugia_Schwab_2020}.

\item The extension to non-piecewise-constant potentials $V$, in particular to smooth potentials that are relevant for applications (e.g.\ Coulomb-interaction potentials in quantum mechanics).
Useful tools towards this goal are the ``generalized plane waves'' developed for 
the Helmholtz equation in \cite{ImbertGerard_Despres_2014} and subsequent 
papers, and the closely related ``quasi-Trefftz'' approach developed for the 
wave equation in~\cite{ImbertGerard_Moiola_Stocker}.

\item Efficient implementations, e.g.\ using the closed-form integration 
of~\cite[\S4.1]{Hiptmair_Moiola_Perugia_2016}.

\item A more accurate analysis of the Trefftz discrete spaces in order to optimise the choice of the parameters ($k_\ell$ and $\bd_\ell$ in \eqref{EQN::BASIS-FUNCTIONS}), improving conditioning, robustness and accuracy.
Basis different from complex exponentials might also be devised and analysed.
The poor conditioning of the time-stepping matrix is likely to be the main bottleneck for the use of the proposed scheme in demanding applications: the Trefftz technologies recently developed for time-harmonic wave problems might greatly help under this respect.

\item The extension to initial boundary value problems with non-reflecting boundary conditions, which are often used to truncate unbounded domains (e.g.\ \cite{Antonie_Besse_Mouysset_2004}).
Under this respect the Trefftz approach is promising as it {allows for} the selection of outward-propagating basis functions on boundary cells, as in \cite{Egger_Kretzchmar_Schnepp_Tsukerman_Weiland_2015}.
\end{itemize}

\bibliographystyle{plain}

\end{document}